\documentclass[11pt]{amsart}
\usepackage{amsmath,amssymb,mathrsfs,mathtools}
\usepackage{xcolor}
\usepackage[colorlinks=true,citecolor=blue,urlcolor=blue]{hyperref}
\usepackage{graphicx}
\newtheorem{thm}{Theorem}[section]
\newtheorem{prop}[thm]{Proposition}
\newtheorem{cor}[thm]{Corollary}
\theoremstyle{definition}
\newtheorem{definition}[thm]{Definition}
\theoremstyle{remark}
\newtheorem{remark}[thm]{Remark}

\title[Almost Hermitian Geometry and Generalized Ricci Flow]{Almost Hermitian Geometry and Generalized Ricci Flow}
\author{Anna Fino} \address[Anna Fino]{Dipartimento di Matematica ``G. Peano'', Universit\`{a} degli studi di Torino \\
Via Carlo Alberto 10\\
10123 Torino, Italy\\
\& Department of Mathematics and Statistics, Florida International University\\
Miami, FL 33199, United States}
\email{annamaria.fino@unito.it, afino@fiu.edu} \author{Julieth Saavedra} \address[Julieth Saavedra]{Universitat Internacional de Catalunya, Immaculada 22, Barcelona, Catalonia 08017}{}\email{jpsaavedra@uic.es}

\begin{document}
\begin{abstract}  

We study almost Hermitian manifolds whose lowered Nijenhuis tensor is totally skew-symmetric and whose characteristic (Bismut) torsion is closed, a condition  known as
the almost-SKT condition. This is  a natural extension of the SKT (pluriclosed) condition beyond the integrable setting. For a fixed almost complex structure, these constraints are linear in the metric. 
We introduce an overdetermined ellipticity condition  for the non-integrable setting which, on compact
manifolds, ensures that
the space of admissible infinitesimal deformations is finite-dimensional,
and prove short-time existence for a constrained Bismut--Ricci flow.

In real dimension six, the almost-SKT condition yields  the following
dichotomy: either the almost complex structure is integrable, or the
skew Nijenhuis tensor is nowhere vanishing and has constant norm, and
the Bismut  Ricci form vanishes identically. Consequently,
on a compact non-integrable almost-SKT six-manifold the constrained
Bismut--Ricci flow is stationary. In real dimension eight, the
overdetermined ellipticity condition always fails, while in higher
dimensions, Nijenhuis rigidity holds on an open dense algebraic locus.

 In arbitrary dimension, we characterize when the requirement that the lowered Nijenhuis tensor be totally skew-symmetric determines the compatible metric pointwise up to scale. On a compact manifold, this algebraic rigidity, together with the closed-torsion condition, implies that the almost-SKT metric compatible with the fixed almost complex structure is unique up to constant scaling. In this case, the constrained Bismut–Ricci flow is an explicit homothetic solution whose behavior is
determined by a single constant.
 We also study the relation with generalized Ricci flow and characterize
the infinitesimal obstruction to lifting the constrained evolution to
generalized Ricci flow when the almost complex structure is allowed to vary.

 \sloppy{Finally, we construct non-integrable almost-SKT structures with nowhere vanishing Nijenhuis tensor and non-flat} characteristic connection on several compact six-manifolds, including $S^1\times S^2\times S^3$ and $S^3\times S^3$.  In these families, the Riemannian metric is fixed and the compatible almost complex structures form an $S^2$-family. By the six-dimensional dichotomy, these structures  have vanishing Bismut Ricci form and are therefore almost-CYT. 
\end{abstract}
\maketitle

\section{Introduction}
  Let $(M^{2n},J,g)$ be an almost Hermitian manifold of complex dimension $n$,
with fundamental form
$
\omega(\cdot,\cdot)=g(\cdot,J\cdot).$
The lowered Nijenhuis tensor is defined by
\[
\widetilde N_g(X,Y,Z)=g(N_J(X,Y),Z).
\]
When emphasizing the dependence of $g$ on $\omega$, we also write
$\widetilde N_g$ as $\widetilde N_\omega$. We assume that this tensor is a
$3$-form and refer to it as the skew Nijenhuis form. Almost Hermitian structures satisfying this condition form the
Gray--Hervella class
$
\mathcal G_1=\mathcal W_1\oplus\mathcal W_3\oplus\mathcal W_4
$ 
\cite{GrayHervella1980,FriedrichIvanov2002}. For such a  structure, there exists a unique Hermitian connection $\nabla^B$ with totally skew-symmetric torsion $T^B$\cite{FriedrichIvanov2002}, called the characteristic connection and also known as the Bismut connection \cite{Bismut1989, Gauduchon1997}. 

In the paper, we consider almost Hermitian structures for which $\widetilde N_g$ is a $3$-form  and   $dT^B=0.$ This is the almost-SKT condition used in \cite{FinoFowdar2026} and when $J$ is integrable, it reduces to the pluriclosed, or SKT, condition. Throughout, manifolds are assumed to be connected and without boundary
unless otherwise stated.

The purpose  is to study deformations of these structures and a natural Bismut--Ricci flow within this constrained class. The key observation is that when $J$ is fixed, both constraints are linear in the associated  fundamental form $\omega.$ The infinitesimal deformations satisfying both constraints form the kernel of a fixed linear second-order operator $L_J$, defined on the bundle $E_J$ of algebraically admissible Hermitian variations. At each $x\in M$, its principal symbol is
\[
\sigma_2(L_J)(x,\xi)\alpha
   =\xi\wedge J^*\xi\wedge\alpha,
   \qquad
   \xi\in T_x^*M,\quad \alpha\in E_{J,x}.
\]
We prove that if $M$ is compact and  this symbol is injective for every $x\in M$ and every nonzero $\xi\in T_x^*M$, a condition denoted by \emph{(OE)}, then the space of admissible deformations is finite-dimensional (Proposition \ref{prop:overdetermined-elliptic-final}). Consequently, the projected Bismut--Ricci flow reduces to an ordinary differential equation on this space.

Fix an almost-SKT  metric $g_0$ with associated fundamental form $\omega_0$, and set
$\mathcal V_J:=\ker L_J.$
We denote by $\mathbb P_J$ the $L^2(g_0)$-orthogonal projection onto $\mathcal V_J$ and by $\rho^B_\omega$ the first Ricci form of the characteristic connection associated with $\omega$. Our first main result establishes short-time existence and a continuation
criterion for equations of the form
$\partial_t\omega=\mathbb P_JF(\omega)$, where $F$ is a smooth
$(1,1)$-form-valued map on the admissible metrics  (Corollary \ref{cor:OE-short-time-final} and Proposition \ref{prop:finite-dimensional-continuation}). In particular, under the constant-rank assumption for the algebraic constraint and condition \emph{(OE)}, the constrained Bismut--Ricci equation
\[
\partial_t\omega
   =-\mathbb P_J\bigl((\rho^B_\omega)^{1,1}\bigr),
   \qquad \omega(0)=\omega_0,
\]
has a unique short-time solution preserving both almost-SKT constraints (Proposition \ref{thm:Bismut-short-time-final}). If its maximal forward existence time $T_{\max}$ is finite, then $\omega(t)$ leaves every compact subset of the positive cone
\[
(\omega_0+\mathcal V_J)\cap\{\omega>0\}
\]
as $t\nearrow T_{\max}$. In particular, a uniform two-sided bound
\[
c\,g_0\leq g(t)\leq C\,g_0
\]
guarantees continuation of the solution.

Our second main result concerns Nijenhuis rigidity. We say that $(J, \omega_0)$ is Nijenhuis-rigid if the algebraically admissible Hermitian variations are pointwise proportional to $\omega_0$, and characterize this condition by the orthogonal indecomposability of the complex $(0,3)$-part $\Psi$ of the skew Nijenhuis form (Proposition \ref{prop:algebraic-rigidity-criterion}). This yields a sharp dimensional picture: in complex dimension three every nonzero skew Nijenhuis form is rigid; in complex dimension four condition \emph{(OE)} always fails; and in complex dimension at least five orthogonally indecomposable complex three-forms constitute an open dense algebraic locus (Corollary \ref{cor:dimension-trichotomy}). On a compact manifold, pointwise rigidity implies that the space of positive $J$-compatible almost-SKT forms is the single ray
$$
\{\lambda\omega_0:\lambda>0\}$$ (Theorem \ref{thm:global-rigid-moduli}).
Consequently, the constrained Bismut--Ricci flow has the explicit homothetic solution
\[
\omega(t)=(1+\mu_0t)\omega_0,
\]
whose maximal existence interval and expanding, stationary, or shrinking behavior are determined by the sign of $\mu_0$ (Theorem \ref{thm:rigid-complete-dynamics}). If, in addition, $\omega_0$ is balanced, that is, $d(\omega_0^{n-1})=0$, then $\mu_0$ is determined cohomologically by pairing $c_1(TM,J)$ with $[\omega_0^{n-1}]$; in particular, the flow is stationary whenever $c_1(TM,J)=0$ in $H^2(M,\mathbb R)$  (Corollary \ref{cor:balanced-rigid-topological-slope}).

 In real dimension six, the condition $d T^B =0$ imposes a stronger
rigidity phenomenon. We prove that either $J$ is integrable, or the skew
Nijenhuis tensor is nowhere vanishing and has constant norm; in the latter
case the Bismut Ricci form vanishes identically
(Theorem~\ref{thr3.16-dimsix}). Consequently, every compact non-integrable almost-SKT
six-manifold is Nijenhuis-rigid and every solution of the constrained
Bismut--Ricci flow is stationary (Corollary~\ref{cor3.17-dimsix}).

 We next turn to the relation with generalized Ricci flow
\cite{GarciaFernandez2019,GarciaFernandezJordanStreets2021}.  
 The almost-Hermitian analogue of the Bismut--Hermitian--Einstein condition, i.e., $
dT^B=0, \rho^B=0,
$
implies the steady generalized Ricci soliton equations. In the compact setting, this implication is contained in \cite[Theorem~3.5]{KennonStreets}.
In the compact case, a co-closed Lee form yields a stationary
solution of generalized Ricci flow, generalizing
\cite[Lemma~2.9]{ABLS} to strong ABHE structures.
In real dimension six, structures with skew Nijenhuis tensor and $\rho^B=0$ are also called almost-CYT in \cite{FinoFowdar2026}. We also identify the precise obstruction to lifting the constrained Bismut--Ricci flow to generalized Ricci flow when the almost complex structure is allowed to vary, and prove that the associated fixed-metric deformation equation preserving the skew-Nijenhuis condition is overdetermined elliptic (Theorem \ref{thm:exact-GRF-lift-constrained} and Proposition \ref{prop:J-G1-overdetermined-elliptic}).

We next construct two families of compact six-dimensional non-integrable
examples. The first construction starts from a compact simply connected strong Sasaki-with-torsion $5$-manifold $M$, in the sense of \cite{BrienzaFinoFowdarGrantcharov2026}, with vanishing characteristic
Ricci form and satisfying a natural non-degeneracy condition. Its horizontal distribution carries a parallel quaternionic triple whose
first member is the original transverse complex structure. Quaternionic
rotation then produces an $S^2$-family of orthogonal almost complex
structures on $S^1\times M$ sharing the same characteristic connection
and closed torsion.
 Applying the first construction to the regular strong $\nabla$-Einstein
five-manifolds $L^{a,b,c}$ considered in
\cite{BrienzaFinoFowdarGrantcharov2026}, which are diffeomorphic to
$S^2\times S^3$, yields { non-Bismut-flat almost-SKT and almost-CYT
structures on the compact six-manifold} $S^1\times S^2\times S^3$ (Theorem \ref{thm:ST-quaternionic-rotation} and Corollary \ref{cor:Labc-quaternionic-examples}).
 The second construction starts from compact simply connected non-K\"ahler
quasi-regular Bismut--Hermitian--Einstein threefolds and produces analogous
families on several compact threefolds,
 
including
$S^3\times S^3$  (Theorem \ref{thm:BHE-threefold-quaternionic-rotation} and  Corollary \ref{cor:BHE-nonflat-families}).
Together, these examples answer the question
raised in \cite{FinoFowdar2026} concerning non-Bismut-flat almost-SKT and
almost-CYT six-manifolds with $N_J\neq0$.
 The stationarity of the compact six-dimensional examples constructed
here is not a feature of the particular constructions. Theorem~\ref{thr3.16-dimsix}
shows that every non-integrable almost-SKT six-manifold has
$\rho^B=0$, and Corollary~\ref{cor3.17-dimsix} therefore implies that, in the compact
case, every solution of the constrained Bismut--Ricci flow is stationary.
Thus genuinely non-stationary non-integrable examples, if they exist in
the overdetermined-elliptic regime, must occur in higher dimension.

 Other curvature flows on almost Hermitian manifolds have also been studied. In particular, in \cite{Vezzoni2011}  the Hermitian curvature flow \cite{StreetsTian2011} was extended  to the almost complex setting, using the Chern connection and its torsion. A different almost-Hermitian extension of the pluriclosed flow was introduced in \cite{Kawamura2017,Kawamura2019}.
It evolves arbitrary, almost Hermitian metrics by a strictly parabolic equation built from the Chern curvature. By contrast, our flow is restricted to the characteristic-torsion locus, is driven by the characteristic Ricci form, and, under condition \emph{(OE)}, reduces to a finite-dimensional ordinary differential equation. When $J$ is integrable, both Kawamura's flow and the unprojected equation underlying our construction agree with the classical pluriclosed flow \cite{StreetsTian2010},
\[
\partial_t\omega=-(\rho^B_\omega)^{1,1}.
\]
Thus, our framework exhibits two distinct regimes: a finite-dimensional ODE on the non-integrable locus satisfying \emph{(OE)}, and the infinite-dimensional parabolic flow in the integrable case.

The paper is organized as follows. Section 2 introduces the constraint operator $L_J$ and condition \emph{(OE)}, and proves short-time existence for the constrained Bismut--Ricci flow.  Section 3 characterizes Nijenhuis rigidity, describes the resulting
dimensional phenomena, and proves the six-dimensional dichotomy and
stationarity results.
Section 4 relates strong almost Bismut--Hermitian--Einstein structures to generalized Ricci solitons. Section 5 develops the examples and quaternionic-rotation constructions. Section 6 studies the obstruction to lifting the constrained flow to generalized Ricci flow with moving $J$. Finally, Section 7 discusses further questions.
%%%%%%%%%%%%%

\section{A constrained Bismut--Ricci flow}
\label{sec:constrained-almost-SKT-flow}

Throughout this section,             $M^{2n}$ is compact, the almost complex
structure $J$ is fixed, and  $\omega$ denotes the fundamental form of $(J,g)$.
  In particular, for fixed $J$ and
$\alpha=\partial_t\omega$ the corresponding metric variation is
$h=\partial_tg$ with
$h(X,Y)=-\alpha(X,JY)$. 
We use the sign conventions for $d^c$ and $N_J$ for which the torsion of the
characteristic (Bismut) connection is
\begin{equation}\label{eq:Bismut-torsion-convention}
T^B_\omega=d^c\omega+\widetilde N_\omega,
\qquad
\widetilde N_\omega(X,Y,Z)=g_\omega(N_J(X,Y),Z),
\end{equation}
(so $\widetilde N_\omega=\widetilde N_g$ in the notation of the
introduction, once $g=g_\omega$).
The characteristic Hermitian connection with totally skew-symmetric torsion (also called the Bismut connection)
exists precisely when $\widetilde N_\omega$ is a  $3$-form; see
\cite{FriedrichIvanov2002}.  We denote by
\[
\rho^B_\omega(X,Y)
=
\frac12 R^B_\omega(X,Y,e_a,Je_a)
\]
the first Ricci form of the characteristic connection, where
$\{e_a\}$ is a local $g_\omega$-orthonormal frame.

We consider the two constraints
\begin{equation}\label{eq:almost-SKT-constraint}
\widetilde N_\omega\in\Omega^3(M),
\qquad
dT^B_\omega=0.
\end{equation}
When $J$ is integrable, $N_J=0$ and the second condition is the usual
pluriclosed condition.

\subsection{The two linear constraints}

Let $S_J^2T^*M$ be the bundle of $J$-invariant symmetric two-tensors.  For
$h\in\Gamma(S_J^2T^*M)$ define
\begin{equation}\label{eq:B-N-definition-final}
\mathcal B_{N_J}(h)(X,Y,Z)
:=h(N_J(X,Y),Z)+h(N_J(X,Z),Y).
\end{equation}
If $\widetilde N_\omega$ is skew-symmetric and $\partial_tg=h$, then
\[
\frac{d}{dt}\left[
\widetilde N_\omega(X,Y,Z)+\widetilde N_\omega(X,Z,Y)
\right]
=\mathcal B_{N_J}(h)(X,Y,Z).
\]
Thus the infinitesimal Hermitian variations preserving the first constraint
are exactly $\ker\mathcal B_{N_J}$.

We use the fixed identification
\begin{equation}\label{eq:I-final}
\mathcal I:S_J^2T^*M\longrightarrow\Lambda_J^{1,1}T^*M,
\qquad
\mathcal I(h)(X,Y)=h(X,JY),
\end{equation}
whose inverse is $\mathcal I^{-1}(\alpha)(X,Y)=-\alpha(X,JY)$.  Assume for the
moment that $\mathcal B_{N_J}$ has constant pointwise rank and set
\begin{equation}\label{eq:EJ-final}
\mathcal E_J
:=\mathcal I(\ker\mathcal B_{N_J})
\subset\Lambda_J^{1,1}T^*M.
\end{equation}
Then $\mathcal E_J$ is a smooth vector subbundle.  Notice that
$\omega_0\in\Gamma(\mathcal E_J)$ whenever
$\widetilde N_{\omega_0}$ is a $3$-form.

Let $\omega(t)$ be a Hermitian variation and put
$\alpha=\partial_t\omega$.  Since $J$, and hence $N_J$, is fixed,
\begin{align*}
\partial_t\widetilde N_\omega(X,Y,Z)
&=\mathcal I^{-1}(\alpha)(N_J(X,Y),Z)
=-\alpha(N_J(X,Y),JZ)\\
&=\alpha(JN_J(X,Y),Z).
\end{align*}
Define the zeroth-order operator
\begin{equation}\label{eq:Ncal-final}
\mathcal N_J(\alpha)(X,Y,Z)
:=\alpha(JN_J(X,Y),Z).
\end{equation}
Then
\begin{equation}\label{eq:variation-T-final}
\partial_tT^B_\omega=d^c\alpha+\mathcal N_J(\alpha),
\end{equation}
and therefore
 \begin{equation}\label{eq:variation-dT-final}
\partial_t(dT^B_\omega)=\mathscr D_J\alpha,
\qquad
\mathscr D_J\alpha
:=d\bigl(d^c\alpha+\mathcal N_J(\alpha)\bigr)
=dd^c\alpha+d\mathcal N_J(\alpha).
\end{equation}
Because $J$ is fixed, $\mathscr D_J$ is a fixed linear differential operator
of order two.  We define
\begin{equation}\label{eq:LJ-final}
L_J:=\mathscr D_J\big|_{\Gamma(\mathcal E_J)},
\qquad
\mathscr V_J:=\ker L_J.
\end{equation}
The space $\mathscr V_J$ is precisely the space of infinitesimal variations
which preserve both constraints.

\begin{prop}
\label{prop:constraint-preservation-final}
Let $\omega(t)$ be a smooth family of positive real $(1,1)$-forms with fixed
$J$ and assume
$$\partial_t\omega(t)\in\mathscr V_J$$
for every $t$.  
If $\omega(0)$ satisfies \eqref{eq:almost-SKT-constraint}, then
\[
\widetilde N_{\omega(t)}\in\Omega^3(M),
\qquad
dT^B_{\omega(t)}=0
\]
for every time for which the family exists.
\end{prop}

\begin{proof}
The first assertion follows by differentiating the skew-symmetry defect and
using $\partial_t\omega\in\Gamma(\mathcal E_J)$.  The second follows from
\eqref{eq:variation-dT-final}:
\[
\partial_t(dT^B_\omega)=L_J(\partial_t\omega)=0.
\]
Both defects therefore remain equal to their initial value, which is zero.
\end{proof}

In fact, the same argument shows something slightly stronger.  Since $J$ is
fixed, both constraint maps are linear in $\omega$.  Hence every positive form
in the affine space
\begin{equation}\label{eq:affine-constraint-space}
\omega_0+\mathscr V_J
\end{equation}
automatically satisfies \eqref{eq:almost-SKT-constraint} whenever $\omega_0$
does.

\subsection{The principal symbol and an overdetermined ellipticity criterion}

The Nijenhuis term in \eqref{eq:variation-dT-final} has order at most one after
applying $d$.  Thus the second-order principal symbol of $L_J$ is, up to the
irrelevant sign determined by the convention for $d^c$,
\begin{equation}\label{eq:symbol-LJ-final}
\sigma_2(L_J)(x,\xi)\alpha
=
\xi\wedge J^*\xi\wedge\alpha,
\qquad
\alpha\in\mathcal E_{J,x}.
\end{equation}

We shall use the following pointwise condition.

\medskip

 \begin{definition}
The pair \((J,\omega_0)\) is said to satisfy the
\emph{overdetermined ellipticity condition} \emph{(OE)} if, for every
\(x\in M\) and every \(0\neq\xi\in T_x^*M\),
\begin{equation}\label{eq:OE-final}
\alpha\in\mathcal E_{J,x},
\qquad
\xi\wedge J^*\xi\wedge\alpha=0
\quad\Longrightarrow\quad
\alpha=0.
\end{equation}
\end{definition}

 Let
$L_J^{*,0}\colon\Omega^4(M)\longrightarrow\Gamma(E_J)$
denote the formal $L^2(g_0)$-adjoint of $L_J$.
\begin{prop}[Finite dimensionality of the admissible space]
\label{prop:overdetermined-elliptic-final}
Assume that $\mathcal E_J$ is a smooth vector bundle and that \emph{(OE)}
holds.  Then $L_J$ is overdetermined elliptic.  Consequently
$L_J^{*,0}L_J$ is an elliptic operator of order four on $\mathcal E_J$,
$\mathscr V_J=\ker L_J$ is finite dimensional and consists of smooth forms,
and the $L^2(g_0)$-orthogonal projection
\[
\mathbb P_J:L^2(\Lambda_J^{1,1})\longrightarrow\mathscr V_J
\]
is a finite-rank smoothing operator.
\end{prop}

\begin{proof}
By \eqref{eq:symbol-LJ-final} and \emph{(OE)}, the symbol of $L_J$ is injective
for every nonzero covector.  Hence
\[
\sigma_4(L_J^{*,0}L_J)(x,\xi)
=
\sigma_2(L_J)(x,\xi)^{*,0}\sigma_2(L_J)(x,\xi)
\]
is positive definite on $\mathcal E_{J,x}$ for $\xi\neq0$.  Thus
$L_J^{*,0}L_J$ is elliptic.  Standard elliptic theory on the compact manifold
$M$ gives finite dimensionality and smoothness of its kernel.  Since
\[
\ker(L_J^{*,0}L_J)=\ker L_J,
\]
the same holds for $\mathscr V_J$.  Orthogonal projection onto a
finite-dimensional space of smooth sections has a smooth integral kernel and
is therefore smoothing.
\end{proof}

 Note that \emph{(OE)} fails at every point where $\widetilde N$ vanishes
(Proposition~\ref{prop:N-kernel-obstructs-OE}); in particular it never holds for
integrable $J$, in accordance with the infinite-dimensionality of the
pluriclosed variational space (Remark~\ref{rem:pluriclosed-limit}).

\subsection{Existence and continuation for projected evolution equations}

The previous proposition allows one to avoid a nonlocal parabolic projection
argument altogether.  Once $\mathscr V_J$ is finite dimensional, the
constrained flow is an ordinary differential equation on the finite-dimensional
manifold of admissible metrics.

Let $\mathcal U\subset\omega_0+\mathscr V_J$ be the open subset consisting of
positive forms and
$$
\mathcal F:\mathcal U\longrightarrow\Omega_J^{1,1}(M,\mathbb R)
$$
be a smooth map.

\begin{prop}[General short-time existence]
\label{thm:general-short-time-final}
Assume that $\omega_0$ satisfies \eqref{eq:almost-SKT-constraint} and that
$\mathscr V_J$ is finite dimensional and consists of smooth forms.  Then the
projected equation
\begin{equation}\label{eq:general-projected-flow-final}
\partial_t\omega
=
\mathbb P_J\mathcal F(\omega),
\qquad
\omega(0)=\omega_0
\end{equation}
has a unique smooth solution on some interval $[0,\varepsilon)$.  After
shrinking $\varepsilon$ if necessary, $\omega(t)$ remains positive and
satisfies both constraints in \eqref{eq:almost-SKT-constraint}.
\end{prop}

\begin{proof}
Choose an $L^2(g_0)$-orthonormal basis
$\{\phi_1,\ldots,\phi_m\}$ of $\mathscr V_J$.  Every element of $\mathcal U$
near $\omega_0$ can be written uniquely as
\[
\omega(c)=\omega_0+\sum_{a=1}^m c_a\phi_a.
\]
Equation \eqref{eq:general-projected-flow-final} is equivalent to the finite
system
\begin{equation}\label{eq:finite-ODE-final}
\dot c_a
=
\left\langle
\mathcal F(\omega(c)),\phi_a
\right\rangle_{L^2(g_0)},
\qquad
a=1,\ldots,m.
\end{equation}
The right-hand side is a smooth function of $c$ on the open set where
$\omega(c)$ is positive.  The ordinary differential equation theorem gives a
unique solution for short time.  Positivity is open, so it persists after
possibly reducing the time interval.  Finally
$\partial_t\omega\in\mathscr V_J$ by construction, and
Proposition~\ref{prop:constraint-preservation-final} gives preservation of the
two constraints.
\end{proof}

Combining  Proposition~\ref{thm:general-short-time-final}  with
Proposition~\ref{prop:overdetermined-elliptic-final} yields the following
 criterion.

\begin{cor}[Short-time existence under \emph{(OE)}]
\label{cor:OE-short-time-final}
Assume that $\mathcal B_{N_J}$ has constant pointwise rank and that
\emph{(OE)} holds. Then  \eqref{eq:general-projected-flow-final} has a unique smooth
short-time solution  for every smooth map $\mathcal F$ as above, defined on
the space of admissible metrics. 
\end{cor}

\begin{prop}[Maximal continuation criterion]\label{prop:finite-dimensional-continuation}
Assume that $\mathscr V_J$ is finite dimensional and consists of smooth forms,
and let $\omega(t)$ be the maximal solution of
\eqref{eq:general-projected-flow-final} on $[0,T_{\max})$.  If
$T_{\max}<\infty$, then $\omega(t)$ leaves every compact subset of the positive
cone
\[
\mathcal U=(\omega_0+\mathscr V_J)\cap\{\omega>0\}
\]
as $t\nearrow T_{\max}$.  In particular, if for some constants
$0<c<C<\infty$ one has
\[
c\,g_0\le g(t)\le C\,g_0
\qquad\text{on }[0,T_{\max}),
\]
then the solution extends past $T_{\max}$.
\end{prop}

\begin{proof}
In the coordinates $c=(c_1,\ldots,c_m)$ used in
\eqref{eq:finite-ODE-final}, the flow is a smooth autonomous ODE on the open
set corresponding to $\mathcal U$.  The standard continuation theorem for
ordinary differential equations says that a finite maximal time can occur
only if the trajectory leaves every compact subset of that open set.  Since
all norms are equivalent on the finite-dimensional space $\mathscr V_J$, the
uniform two-sided metric bound places the coefficient vector $c(t)$ in a
bounded set whose closure stays inside the positive cone.  Hence the
trajectory is contained in a compact subset of $\mathcal U$ and extends.
\end{proof}

\begin{remark}
Thus, once the constraint is overdetermined elliptic, there is no independent
higher-derivative blow-up mechanism: all spatial derivatives of $\omega(t)$
are controlled by the finitely many coefficients of $\omega(t)-\omega_0$.
A finite-time singularity must therefore be accompanied by degeneration or
unboundedness of the metric inside the finite-dimensional admissible cone.
\end{remark}

\subsection{The constrained Bismut--Ricci flow}

For an admissible metric the characteristic connection exists, hence its
Bismut--Ricci form $\rho^B_\omega$ is defined.  We take
\[
\mathcal F_B(\omega):=-(\rho^B_\omega)^{1,1}
\]
and define
\begin{equation}\label{eq:Bismut-Ricci-constrained-final}
\partial_t\omega
=
-\mathbb P_J\bigl((\rho^B_\omega)^{1,1}\bigr),
\qquad
\omega(0)=\omega_0.
\end{equation}

\begin{prop}[Short-time existence for the constrained Bismut--Ricci flow]
\label{thm:Bismut-short-time-final}
Let $(M^{2n},J,\omega_0)$ be compact and assume that
\eqref{eq:almost-SKT-constraint} holds.  Suppose that
$\mathcal B_{N_J}$ has constant pointwise rank and that \emph{(OE)} holds.
Then there is $\varepsilon>0$ for which
\eqref{eq:Bismut-Ricci-constrained-final} has a unique smooth solution on
$[0,\varepsilon)$.  Along this solution
\[
\widetilde N_{\omega(t)}\in\Omega^3(M),
\qquad
dT^B_{\omega(t)}=0.
\]
\end{prop}

\begin{proof}
Under \emph{(OE)}, Proposition~\ref{prop:overdetermined-elliptic-final} shows
that $\mathscr V_J$ is finite dimensional and smooth.  On the open subset of
positive forms in $\omega_0+\mathscr V_J$, both constraints hold, so the
characteristic connection and $\rho^B_\omega$ are defined.  Since curvature is
a smooth expression in the metric, its inverse, and finitely many derivatives
of the metric, the map
\[
\omega\longmapsto-(\rho^B_\omega)^{1,1}
\]
is smooth on this finite-dimensional open set.  Proposition
\ref{thm:general-short-time-final} therefore applies.
\end{proof}

The important point is that no parabolicity statement for the projected
Bismut--Ricci operator is needed in Proposition
\ref{thm:Bismut-short-time-final}: overdetermined ellipticity of the
constraint makes the admissible space finite dimensional, and the flow
becomes an ODE.
{\begin{remark}
\label{rem:pluriclosed-limit}
If $J$ is integrable, the algebraic constraint in
\eqref{eq:almost-SKT-constraint} is vacuous, $T^B_\omega=d^c\omega$, and the
differential constraint reduces to the pluriclosed condition
$dd^c\omega=0$. The admissible space is then infinite dimensional, but the
projection $\mathbb P_J$ is unnecessary, since the unprojected evolution
\begin{equation}\label{eq:pluriclosed-limit-final}
\partial_t\omega=-(\rho^B_\omega)^{1,1},
\qquad
\omega(0)=\omega_0,
\end{equation}
already preserves the condition $dd^c\omega=0$. Equation
\eqref{eq:pluriclosed-limit-final} is the pluriclosed flow of
\cite{StreetsTian2010}. For pluriclosed initial data it admits a
unique smooth short-time solution by parabolic methods, rather than by the
finite-dimensional argument of
Proposition~\ref{thm:Bismut-short-time-final}. Thus
\eqref{eq:Bismut-Ricci-constrained-final} has the expected integrable limit,
while in the non-integrable setting the characteristic torsion acquires the
Nijenhuis correction and the admissible variations are subject to the
additional algebraic constraint.
\end{remark}
}
\section{Rigidity and the dimensional picture}
  This section studies when the overdetermined ellipticity condition
\emph{(OE)} holds pointwise. We introduce Nijenhuis rigidity,
characterize it algebraically through the orthogonal indecomposability
of the complex Nijenhuis three-form, and derive the resulting
dimensional picture. We then specialize to real dimension six, where
pointwise rigidity holds wherever the skew Nijenhuis form is nonzero
and the condition $dT^B=0$ yields the global dichotomy of
Theorem~\ref{thr3.16-dimsix}. We also rule out standard parallel-torsion families as a
source of non-flat non-integrable examples.

\subsection{A rigidity criterion}

\begin{sloppypar} Throughout this subsection we assume $n\geq 2$.
A particularly transparent situation occurs when the algebraic Nijenhuis
constraint determines the Hermitian metric infinitesimally up to scale.
\end{sloppypar}

\begin{definition}
We say that $(J,\omega_0)$ is \emph{Nijenhuis-rigid} if
\begin{equation}\label{eq:N-rigid-final}
\mathcal E_{J,x}=\mathbb R\,\omega_{0,x}
\qquad\text{for every }x\in M.
\end{equation}
\end{definition}

\begin{prop}[Rigidity implies \emph{(OE)}]
\label{prop:rigidity-OE-final}
Assume $n\ge2$ and \eqref{eq:N-rigid-final}.  Then \emph{(OE)} holds.
\end{prop}

\begin{proof}
Let $0\neq\xi\in T_x^*M$.  Choose a $g_0$-unitary real coframe so that,
up to a nonzero scalar, $\xi=e^1$ and $J^*\xi=e^2$.  Then
\[
\omega_0=e^{12}+e^{34}+\cdots+e^{2n-1,2n},
\]
and hence
\[
\xi\wedge J^*\xi\wedge\omega_0
=e^{12}\wedge
\left(e^{34}+\cdots+e^{2n-1,2n}\right)\neq0.
\]
Since every $\alpha\in\mathcal E_{J,x}$ is a scalar multiple of $\omega_0$,
\eqref{eq:OE-final} follows.
\end{proof}

For a closed-torsion initial structure, rigidity gives more: the global
admissible space is exactly one-dimensional.

\begin{prop}
\label{prop:rigid-kernel-final}
Assume that $M$ is compact, $dT^B_{\omega_0}=0$, and
\eqref{eq:N-rigid-final} holds.  Then
\begin{equation}\label{eq:V-rigid-final}
\mathscr V_J=\mathbb R\,\omega_0.
\end{equation}
\end{prop}

\begin{proof}
Every section of $\mathcal E_J$ has the form $\alpha=f\omega_0$.  Since
$\mathcal N_J(\omega_0)=\widetilde N_{\omega_0}$, we obtain
\begin{align}
L_J(f\omega_0)
&=d\left(d^c(f\omega_0)+f\widetilde N_{\omega_0}\right)\nonumber\\
&=d\left(d^cf\wedge\omega_0+fT^B_{\omega_0}\right)\nonumber\\
&=dd^cf\wedge\omega_0
-d^cf\wedge d\omega_0
+df\wedge T^B_{\omega_0},
\label{eq:L-scalar-final}
\end{align}
where the term $f\,dT^B_{\omega_0}$ vanishes by hypothesis.
If $L_J(f\omega_0)=0$, wedge \eqref{eq:L-scalar-final} with
$\omega_0^{n-2}$.  After division by the volume form this gives a scalar
second-order equation
\begin{equation}\label{eq:scalar-elliptic-final}
P_0f=0
\end{equation}
with no zeroth-order term.  Its principal symbol is a nonzero fixed multiple of
\[
\frac{\xi\wedge J^*\xi\wedge\omega_0^{n-1}}
{dV_{g_0}}
=c_n|\xi|_{g_0}^2,
\qquad c_n>0,
\]
so, after multiplying by an overall sign if necessary, $P_0$ is uniformly
elliptic.  The strong maximum principle on the compact connected manifold
therefore implies that $f$ is constant.  Conversely, constant multiples of
$\omega_0$ clearly lie in $\ker L_J$, because $dT^B_{\omega_0}=0$.  This proves
\eqref{eq:V-rigid-final}.
\end{proof}

\begin{thm}[Global rigidity of the almost-SKT metric]\label{thm:global-rigid-moduli}
Assume that $M$ is compact, that $\omega_0$ satisfies
\eqref{eq:almost-SKT-constraint}, and that $(J,\omega_0)$ is
Nijenhuis-rigid.  Let
\[
\mathcal A_J
:=\left\{\omega>0:\ \widetilde N_\omega\in\Omega^3(M),\quad
 dT^B_\omega=0\right\}
\]
be the space of positive $J$-compatible almost-SKT forms with $J$ fixed.
Then
\begin{equation}\label{eq:global-rigid-moduli}
\mathcal A_J=\{\lambda\omega_0:\lambda>0\}.
\end{equation}
Thus the compatible almost-SKT metric is globally unique up to constant
homothety, not merely infinitesimally rigid.
\end{thm}

\begin{proof}
For fixed $J$, both the skew-Nijenhuis condition and the closed-torsion
condition are linear in the fundamental form.  Hence, if $\omega\in\mathcal
A_J$, then $\alpha:=\omega-\omega_0$ belongs to $\mathscr V_J$.  By
Proposition~\ref{prop:rigid-kernel-final}, $\alpha=c\omega_0$ for a constant
$c$.  Positivity gives $1+c>0$, proving \eqref{eq:global-rigid-moduli}.  The
converse follows from linearity of the torsion and the Nijenhuis lowering
under constant homothety.
\end{proof}

\begin{thm}[Complete dynamics in the Nijenhuis-rigid case]
\label{thm:rigid-complete-dynamics}
Under the hypotheses of Theorem~\ref{thm:global-rigid-moduli},  the unique
solution of the constrained Bismut--Ricci equation
\eqref{eq:Bismut-Ricci-constrained-final} is given explicitly by
\begin{equation}\label{eq:rigid-homothetic-all-dim}
\omega(t)=\lambda(t)\omega_0,\qquad
\lambda(t)=1+\mu_0t,
\end{equation}
where
\begin{equation}\label{eq:rigid-mu-all-dim}
\mu_0
:=-\frac{\displaystyle\int_M
\left\langle(\rho^B_{\omega_0})^{1,1},\omega_0\right\rangle_{g_0}\,dV_{g_0}}
{\displaystyle\int_M|\omega_0|_{g_0}^2\,dV_{g_0}}.
\end{equation}
Its maximal time interval containing $0$ is exactly
\begin{equation}\label{eq:rigid-maximal-interval}
I_{\max}=\{t\in\mathbb R:1+\mu_0t>0\}.
\end{equation}
In particular, the forward solution is immortal for $\mu_0\ge0$, while for
$\mu_0<0$ it becomes extinct at
$T_{\max}=-\mu_0^{-1}.$

After volume normalization the solution is stationary:
\begin{equation}\label{eq:rigid-normalized-stationary}
\left(\frac{\operatorname{Vol}(g_0)}{\operatorname{Vol}(g(t))}\right)^{1/n}
\omega(t)=\omega_0.
\end{equation}
Moreover
\begin{equation}\label{eq:rigid-scaling-laws}
\operatorname{Vol}(g(t))=\lambda(t)^n\operatorname{Vol}(g_0),
\qquad
|T^B(t)|_{g(t)}^2=\lambda(t)^{-1}|T^B_0|_{g_0}^2.
\end{equation}
If the $(0,4)$ curvature tensor of the characteristic connection is nonzero,
then
\[
|\operatorname{Rm}^B(g(t))|_{g(t)}
=\lambda(t)^{-1}|\operatorname{Rm}^B(g_0)|_{g_0}.
\]
Thus every finite-time extinction in the rigid class is of type I with
respect to the characteristic curvature.
\end{thm}

\begin{proof}
Theorem~\ref{thm:global-rigid-moduli} and
Proposition~\ref{prop:rigid-kernel-final} imply that every solution is of the
form $\omega(t)=\lambda(t)\omega_0$.  Under a spatially constant homothety
$g\mapsto\lambda g$, the Levi--Civita connection is unchanged and $T^B$ as a
three-form scales by $\lambda$.  Raising one index therefore leaves the
torsion endomorphism in the characteristic connection unchanged.  Hence the
characteristic connection, its curvature endomorphism and its Ricci form are
unchanged.  Since the $L^2(g_0)$ projection is fixed, the flow equation reduces
to $\dot\lambda=\mu_0$, proving \eqref{eq:rigid-homothetic-all-dim}.  Positivity
is equivalent to $\lambda>0$, which proves the exact maximal interval.
The remaining identities follow directly from constant rescaling.  For the
curvature statement, the $(0,4)$ curvature tensor acquires one factor
$\lambda$ from lowering its output index while its norm uses four inverse
metrics, hence its norm scales by $\lambda^{-1}$.  When $\mu_0<0$,
$\lambda(t)=|\mu_0|(T_{\max}-t)$, giving the type-I rate.
\end{proof}
 When $\omega_0$ is balanced, the constant $\mu_0$ governing the
homothetic flow admits a cohomological interpretation in terms of the
real first Chern class.

\begin{cor}
\label{cor:balanced-rigid-topological-slope}
Under the hypotheses of Theorem~\ref{thm:rigid-complete-dynamics}, assume in addition that $d(\omega_0^{n-1})=0$.  Define the $\omega_0$-degree
of the real first Chern class by
\begin{equation}\label{eq:balanced-c1-degree}
\deg_{\omega_0}c_1(TM,J)
:=\frac1{2\pi}\int_M
\left\langle(\rho^B_{\omega_0})^{1,1},\omega_0\right\rangle_{g_0}\,dV_{g_0}.
\end{equation}
Then this number depends only on $c_1(TM,J)$ and the cohomology class of
$\omega_0^{n-1}$, and
\begin{equation}\label{eq:mu-degree}
\mu_0=-\frac{2\pi}{n\,\operatorname{Vol}(g_0)}
\deg_{\omega_0}c_1(TM,J).
\end{equation}
In particular, if $c_1(TM,J)=0$ in $H^2(M,\mathbb R)$, then the constrained
Bismut--Ricci flow is stationary.
\end{cor}

\begin{proof}
The characteristic connection is unitary, so its first Ricci form is closed
and represents $2\pi c_1(TM,J)$ by Chern--Weil theory.  Since
$d(\omega_0^{n-1})=0$, pairing this closed two-form with
$\omega_0^{n-1}$ is cohomological; the algebraic Hermitian contraction
identity gives \eqref{eq:balanced-c1-degree}, with the sign absorbed by our
fixed fundamental-form convention.  Finally $|\omega_0|_{g_0}^2=n$, so the
denominator in \eqref{eq:rigid-mu-all-dim} is
$n\operatorname{Vol}(g_0)$, giving \eqref{eq:mu-degree}.
\end{proof}

\subsection{Algebraic characterization of Nijenhuis-rigidity}

The rigidity hypothesis above admits a useful intrinsic characterization in
terms of the $(0,3)$-part of the skew Nijenhuis form.  Let $(W,\langle\cdot,\cdot\rangle)$
be a Hermitian vector space of complex dimension $n$, and let
$0\neq\Psi\in\Lambda^3W^*$.  We say that $\Psi$ is \emph{orthogonally
decomposable} if there is a non-trivial Hermitian orthogonal splitting
\[
W=W_1\oplus W_2
\]
and forms $\Psi_i\in\Lambda^3W_i^*$, possibly with one of the $\Psi_i$ equal
to zero, such that
\[
\Psi=\Psi_1+\Psi_2.
\]
Otherwise $\Psi$ is called \emph{orthogonally indecomposable}.

\begin{prop}[Algebraic rigidity criterion]\label{prop:algebraic-rigidity-criterion}
At a point where the skew Nijenhuis form is nonzero, let
$
\Psi:=(\widetilde N_{\omega_0})^{0,3}.$
Then
\[
\mathcal E_{J,x}=\mathbb R\,\omega_{0,x}
\quad\Longleftrightarrow\quad
\Psi_x\text{ is orthogonally indecomposable}.
\]
\end{prop}

\begin{proof}
Write an admissible Hermitian variation as
$h(X,Y)=g_0(AX,Y)$, where $A$ is $g_0$-self-adjoint and commutes with $J$.
Choose a unitary eigenbasis of $A$ on $W=T_x^{0,1}M$, say
$Ae_i=\lambda_i e_i$.  The condition that
$h(N_J(\cdot,\cdot),\cdot)$ be alternating is equivalent to the alternation of
\[
(X,Y,Z)\longmapsto \Psi(X,Y,AZ).
\]
Hence, whenever $\Psi_{ijk}:=\Psi(e_i,e_j,e_k)\neq0$, interchange of the
second and third entries gives
\[
\lambda_k\Psi_{ijk}=\lambda_j\Psi_{ijk}.
\]
Cycling the three entries yields
\[
\Psi_{ijk}\neq0\quad\Longrightarrow\quad
\lambda_i=\lambda_j=\lambda_k.
\]
\begin{sloppypar}
Thus every nonzero component of $\Psi$ is supported entirely inside one
eigenspace of $A$.  If $A$ has at least two distinct eigenvalues, the
orthogonal eigenspace decomposition gives a non-trivial orthogonal
decomposition of $\Psi$.
\end{sloppypar}

Conversely, if $\Psi=\Psi_1+\Psi_2$ for an orthogonal splitting
$W=W_1\oplus W_2$, choose distinct real numbers $a\neq b$ and let $A=aI$ on
$W_1$ and $A=bI$ on $W_2$.  Then $\Psi(X,Y,AZ)$ is still alternating, so $A$
defines a non-scalar admissible Hermitian variation.  Therefore the only
admissible Hermitian endomorphisms are scalars exactly when $\Psi$ is
orthogonally indecomposable.
\end{proof}

 A nontrivial kernel of the contraction map $v\mapsto\iota_v\Psi$ already
obstructs \emph{(OE)}, and so does $\Psi=0$ itself.
\begin{prop}
\label{prop:N-kernel-obstructs-OE}
Let $x\in M$ and let
$\Psi=(\widetilde N_{\omega_0})^{0,3}_x$.  If either $\Psi=0$ or the contraction
map
\[
T_x^{0,1}M\longrightarrow\Lambda^2(T_x^{0,1}M)^*,
\qquad v\longmapsto\iota_v\Psi
\]
has nontrivial kernel, then \emph{(OE)} fails at $x$.
Consequently, \emph{(OE)} forces the complex Nijenhuis three-form to be
multisymplectic  at every point, i.e. to have injective contraction map
$v\mapsto\iota_v\Psi$.
\end{prop}

\begin{proof}
If $\Psi=0$, the algebraic skew-Nijenhuis constraint imposes no restriction at
$x$, so $\mathcal E_{J,x}=\Lambda_J^{1,1}T_x^*M$ and the form
$\alpha=\xi\wedge J^*\xi$ violates \emph{(OE)}.

Suppose instead that $0\neq v\in\ker\Psi$ and let $K=\mathbb Cv$.  Since
$\Psi$ is alternating, it vanishes whenever one argument lies in $K$.  Let
$A$ be the Hermitian orthogonal projection onto $K$.  Then
$\Psi(X,Y,AZ)=0$ for all $X,Y,Z$, so the Hermitian symmetric tensor
$h(\cdot,\cdot)=g_0(A\cdot,\cdot)$ is algebraically admissible.  The associated
real $(1,1)$-form $\alpha=\mathcal I(h)$ is supported on the real two-plane
underlying $K$.  Choosing a nonzero real covector $\xi$ on that plane gives,
up to a nonzero scalar, $\alpha=\xi\wedge J^*\xi$, and therefore
\[
\xi\wedge J^*\xi\wedge\alpha=0.
\]
Thus \emph{(OE)} fails.
\end{proof}

The ellipticity condition is strictly weaker than Nijenhuis-rigidity.
\begin{prop}
\label{prop:OE-not-rigidity}
In complex dimension six there are algebraically non-rigid skew Nijenhuis forms for
which \emph{(OE)} nevertheless holds.  More precisely, let
\[
W=W_1\oplus W_2,
\qquad \dim_{\mathbb C}W_1=\dim_{\mathbb C}W_2=3,
\]
be an orthogonal Hermitian splitting and let
\[
\Psi=\Psi_1+\Psi_2,
\qquad 0\neq\Psi_i\in\Lambda^3W_i^*.
\]
Then \(\Psi\) is orthogonally decomposable, hence the corresponding algebraic
skew-Nijenhuis constraint is not rigid, while \emph{(OE)} holds.
\end{prop}

\begin{proof}
 The alternation condition on mixed triples forces every admissible
Hermitian endomorphism to preserve each summand.
By Proposition~\ref{prop:algebraic-rigidity-criterion}, its restriction to each summand is scalar.
Conversely, every endomorphism that is scalar on each summand
is admissible.

Equivalently,
for the associated real fundamental forms \(\omega_i\) one has
\[
\mathcal E_J=\mathbb R\omega_1\oplus\mathbb R\omega_2.
\]
Let
\[
\alpha=a\omega_1+b\omega_2,
\qquad 0\neq\xi=x+y,
\]
where \(x\in W_{1,\mathbb R}^*\) and \(y\in W_{2,\mathbb R}^*\).  Suppose
\[
\xi\wedge J^*\xi\wedge\alpha=0.
\]
If \(x\neq0\), the component of exterior degree \((4,0)\) with respect to the
splitting \(W_1\oplus W_2\) is
\[
a\,x\wedge J^*x\wedge\omega_1.
\]
Since \(W_1\) has real dimension six and \(\omega_1\) is symplectic, this form is
nonzero unless $a=0$.  Similarly, if $y\neq0$, the $(0,4)$-component forces
$b=0$.  If both $x$ and $y$ are nonzero, this already gives
$a=b=0$.  If, say, $y=0$, then the $(2,2)$-component
$b\,x\wedge J^*x\wedge\omega_2$
is nonzero unless $b=0$; the case $x=0$ is analogous.  Thus $\alpha=0$
for every $0\neq\xi$, proving \emph{(OE)}.
\end{proof}

\begin{remark}
\begin{sloppypar}
Proposition~\ref{prop:OE-not-rigidity}  shows that condition \emph{(OE)} is strictly weaker
than pointwise Nijenhuis rigidity. In particular, \emph{(OE)} does
not force the bundle of algebraically admissible variations to
have rank one.
\end{sloppypar}
\end{remark}

\begin{cor}
\label{cor:dim8-OE-fails}
\begin{sloppypar}
In complex dimension $n=4$, the condition \emph{(OE)} cannot hold at any
point.  In particular, the overdetermined-elliptic finite-dimensional
existence mechanism of Proposition~\ref{prop:overdetermined-elliptic-final}
never applies to a genuinely non-integrable almost-SKT eight-manifold.
\end{sloppypar}
\end{cor}

\begin{proof}
If $\Psi=0$, apply the first case of
Proposition~\ref{prop:N-kernel-obstructs-OE}.  If $\Psi\neq0$ in complex
dimension four, every three-form has a nontrivial complex one-dimensional
kernel.  The second case of the same proposition applies.
\end{proof}

\begin{cor}[The dimensional picture]\label{cor:dimension-trichotomy}
Let $(M^{2n},J,g)$ be an almost Hermitian manifold and let $\widetilde N$
be its lowered Nijenhuis tensor.  At a point where $\widetilde N$ is
skew-symmetric, the following pointwise statements hold.
\begin{enumerate}
\item If $n\le2$, then $\widetilde N=0$.
\item If $n=3$, every nonzero $\widetilde N$ is Nijenhuis-rigid.
\item If $n=4$, no nonzero $\widetilde N$ is Nijenhuis-rigid, and \emph{(OE)} fails identically.
\item If $n\ge5$, orthogonally indecomposable complex $3$-forms form an open
and dense subset of $\Lambda^3W^*$, with  $W:=T_x^{0,1}M$.   For a fixed splitting type
$k+(n-k)$, the corresponding decomposable locus has real codimension at
least
\[
k(n-k)(n-4).
\]
\end{enumerate}
\end{cor}

\begin{proof}
The first assertion follows from $\Lambda^3W^*=0$ for $n\le2$.  If $n=3$, a
nonzero $3$-form is a complex volume form and cannot split over a non-trivial
orthogonal decomposition, proving the second assertion by
Proposition~\ref{prop:algebraic-rigidity-criterion}.

If $n=4$, every nonzero $\Psi\in\Lambda^3W^*$ has a complex one-dimensional
kernel $K$.  With respect to the Hermitian splitting
$W=K^\perp\oplus K$ one has $\Psi\in\Lambda^3(K^\perp)^*$, so $\Psi$ is
orthogonally decomposable.

 Finally, fix $n\geq 5$ and $1\leq k\leq n-1$.
Orthogonal splittings $W=W_1\oplus W_2$ with
$\dim_{\mathbb C}W_1=k$ are parametrized by the complex
Grassmannian $\operatorname{Gr}_{\mathbb C}(k,W)$, of real
dimension $2k(n-k)$. For each splitting, the space of forms
$\Psi_1+\Psi_2$, with $\Psi_i\in\Lambda^3W_i^*$,
has real dimension
\[
2\binom{k}{3}+2\binom{n-k}{3}.
\]
Thus the corresponding decomposable locus has real codimension
at least
\[
2\binom{n}{3}
-2\binom{k}{3}
-2\binom{n-k}{3}
-2k(n-k)
=
k(n-k)(n-4)>0.
\]
For each splitting type, the decomposable locus on the unit
sphere is the image of a compact sphere bundle and is therefore
closed. Taking the finite union over all splitting types,
we conclude that the indecomposable locus is open and dense.
\end{proof}

\begin{remark}
\label{rem:nonzero-not-enough-final}
The six-dimensional conclusion is special.  In complex dimension $n\ge4$, a
nonzero skew Nijenhuis form need not determine the Hermitian metric up to
scale.  Algebraically, let
$\Psi=dz^1\wedge dz^2\wedge dz^3$ on $\mathbb C^n$.  Its kernel contains the
complex subspace spanned by $\partial_{z^4},\ldots,\partial_{z^n}$.  Hermitian
variations supported on that kernel can preserve the skew condition.  If $\xi$ is a nonzero real covector whose metric dual satisfies
$\iota_{\xi^\sharp}\Psi=0$, one can find a nonzero
such variation $\alpha$ with
\[
\xi\wedge J^*\xi\wedge\alpha=0.
\]
Thus \emph{(OE)} can fail even though the skew Nijenhuis $3$-form is nonzero.
Accordingly, in higher dimension the clean general hypothesis is the symbol
condition \emph{(OE)}, or a stronger algebraic condition such as
Nijenhuis-rigidity, rather than merely $N_J\neq0$.
\end{remark}

\subsection{The six-dimensional non-integrable case}

In real dimension six the Nijenhuis-rigidity condition is automatic wherever
the skew Nijenhuis $3$-form is nonzero.

\begin{prop}[Pointwise six-dimensional rigidity]
\label{lem:dimension-six-rigidity-final}
Let $(V^6,J,g)$ be a Hermitian vector space of real dimension six and let $N$ have the algebraic type
of a Nijenhuis tensor.  Assume
$\widetilde N(X,Y,Z):=g(N(X,Y),Z)$
is a nonzero $3$-form.  If $h$ is a $J$-invariant symmetric bilinear form such
that $h(N(\cdot,\cdot),\cdot)$ is also a $3$-form, then
$h=\lambda g$
for some $\lambda\in\mathbb R$.
\end{prop}

\begin{proof}
Complexify and restrict to $V^{0,1}$.  The Nijenhuis tensor maps
$\Lambda^2V^{0,1}$ to $V^{1,0}$, and the Hermitian metric pairs
$V^{1,0}$ with $V^{0,1}$.  Thus the complexification of $\widetilde N$ has a
$(0,3)$ component $\Psi$.  Since $\dim_{\mathbb C}V^{0,1}=3$ and
$\widetilde N\neq0$, $\Psi$ is a nonzero complex volume form.

Write $h(X,Y)=g(AX,Y)$, where $A$ is $g$-self-adjoint and commutes with $J$.
Choose a unitary eigenbasis $e_1,e_2,e_3$ of $V^{0,1}$,
$Ae_i=\lambda_i e_i$.  The condition that
$h(N(\cdot,\cdot),\cdot)$ is alternating says that
$\Psi(X,Y,AZ)$
is alternating in $X,Y,Z$.  Since
$\Psi(e_1,e_2,e_3)\neq0$, interchange of the second and third slots gives
\[
\lambda_3\Psi(e_1,e_2,e_3)
=-\lambda_2\Psi(e_1,e_3,e_2)
=\lambda_2\Psi(e_1,e_2,e_3),
\]
so $\lambda_2=\lambda_3$.  Cycling the indices gives
$\lambda_1=\lambda_2=\lambda_3$.  Hence $A=\lambda\operatorname{Id}$ and
$h=\lambda g$.
\end{proof}

\begin{remark}
\begin{sloppypar}
The rigidity statement above is compatible with Verbitsky's uniqueness
phenomenon for nearly K\"ahler metrics on an almost complex six-manifold with
nondegenerate Nijenhuis tensor: in that more restrictive nearly K\"ahler
setting the compatible metric is unique up to scale
\cite{Verbitsky2008}.  Proposition~\ref{lem:dimension-six-rigidity-final}
uses only the algebraic skew-Nijenhuis condition and is therefore independent
of the nearly K\"ahler differential equations.
\end{sloppypar}
\end{remark}

 The  condition $dT^B=0$ imposes a further rigidity phenomenon in
real dimension six: either the structure is integrable or the Nijenhuis
tensor is nowhere vanishing.

\begin{thm} \label{thr3.16-dimsix}
Let $(M^6,J,g)$ be an almost Hermitian manifold such that the
lowered Nijenhuis tensor
$
\widetilde N(X,Y,Z)=g(N_J(X,Y),Z)
$
is a three-form and the characteristic torsion satisfies
$dT^B=0.$

Then exactly one of the following alternatives occurs:
\begin{enumerate}
\item $\widetilde N\equiv0$, and hence $J$ is integrable;
\item $\widetilde N$ is nowhere vanishing and has constant pointwise norm.
In this case $\rho^B=0.$
\end{enumerate}
In particular, every  non-integrable almost-SKT six-manifold is
automatically almost-CYT. Moreover, in the non-integrable case the
canonical bundle is trivial.
\end{thm}

\begin{proof}
Set
\[
U:=\{x\in M:\widetilde N_x\neq0\}.
\]
Suppose that $U\neq\varnothing$, and let $U_0$ be a connected component
of $U$.

On $U_0$, the skew Nijenhuis form has type $(3,0)+(0,3)$. Hence its
$(3,0)$-part
$\Psi:=(\widetilde N)^{3,0}$
is a nowhere vanishing smooth section of the canonical bundle
$K_J=\Lambda^{3,0}T^*M$
over $U_0$. After normalization with respect to $g$, $\Psi$ determines
an $\mathrm{SU}(3)$-structure $(\omega,\Omega)$ compatible with $(g,J)$.
We choose $\Omega$ so that $\widetilde N$ is a positive multiple of
$\Omega^+$.
With our convention for the Nijenhuis tensor, Proposition~4.1 of
\cite{FinoFowdar2026} reads
\[
\widetilde N=2\pi_0\Omega^++2\sigma_0\Omega^-.
\]
By the choice of $\Omega$, it follows that
\[
\sigma_0=0,\qquad \pi_0>0,
\]
and hence
$\widetilde N=2\pi_0\Omega^+.$
We note that in \cite{FinoFowdar2026} the characteristic torsion is denoted
by $T_\omega$, and the opposite sign convention is used for the Nijenhuis
three-form.

We next use the condition $dT^B=0$.
 By the $\Lambda^4_6$-component
of the formula for $dT^B$ in
\cite[Theorem~4.4]{FinoFowdar2026}, in the equivalent form obtained
using \cite[(36)]{FinoFowdar2026}, one has
\[
(dT^B)_{\Lambda^4_6}
=
\bigl(
-2\,d\pi_0
+4\sigma_0J\nu_1
-2\sigma_0J\pi_1
\bigr)\wedge\Omega^+.
\]
Since $dT^B=0$ and $\sigma_0=0$, it follows that
\[
d\pi_0\wedge\Omega^+=0.
\]
The $\mathrm{SU}(3)$-equivariant map
\[
\Lambda^1\longrightarrow\Lambda^4_6,
\quad
\alpha\longmapsto\alpha\wedge\Omega^+,
\]
is injective. Therefore
$d\pi_0=0$
on $U_0$.

On the other hand, identity~(36) of \cite{FinoFowdar2026} gives
\[
Jd\sigma_0-d\pi_0
-\pi_0(2\nu_1-\pi_1)
+\sigma_0J(2\nu_1-\pi_1)=0.
\]
Since $\sigma_0=0$ identically on $U_0$, we also have
$d\sigma_0=0$. Together with $d\pi_0=0$, the preceding identity reduces
to
\[
\pi_0(2\nu_1-\pi_1)=0.
\]
As $\pi_0>0$ on $U_0$, we conclude that
$\pi_1=2\nu_1.$
Corollary~4.8 of \cite{FinoFowdar2026} then yields
$\rho^B=0$
on $U_0$.

Since $U_0$ is connected and $d\pi_0=0$, the function $\pi_0$ is a
positive constant on $U_0$. Moreover,
\[
\widetilde N=2\pi_0\Omega^+,
\]
and $\Omega^+$ has constant pointwise norm with respect to the normalized
$\mathrm{SU}(3)$-structure. Hence $|\widetilde N|$ is a positive
constant on $U_0$; denote it by $c>0$.

It remains to show that $U_0=M$. Let $x\in\overline{U_0}$. Since
$|\widetilde N|=c$ on $U_0$, by continuity
\[
|\widetilde N_x|=c>0,
\]
so $x\in U$. Since $U$ is open, there exists a connected neighborhood
$V$ of $x$ contained in $U$. As $x\in\overline{U_0}$, we have
$V\cap U_0\neq\varnothing$. Hence $V$ and $U_0$ lie in the same
connected component of $U$, and therefore $V\subset U_0$. In particular,
$x\in U_0$. Thus $U_0$ is closed in $M$. Since it is also open and
$M$ is connected, $U_0=M$.

We have therefore proved that either $U=\varnothing$, in which case
$\widetilde N\equiv0$ and $J$ is integrable, or $U=M$, in which case
$\widetilde N$ is nowhere vanishing, has constant norm, and
$\rho^B=0.$

Finally, in the latter case
$\Psi=(\widetilde N)^{3,0}$
is a nowhere vanishing global section of the canonical bundle $K_J$.
Hence $K_J$ is trivial.
\end{proof}

As a consequence, every solution of the constrained Bismut--Ricci flow on a compact non-integrable almost-SKT six-manifold is necessarily stationary.

\begin{cor} \label{cor3.17-dimsix}
Let $(M^6,J,\omega_0)$ be compact and almost-SKT, and assume
that $J$ is non-integrable. Then
\[
\mathcal A_J=\{\lambda\omega_0:\lambda>0\},
\qquad
\mathcal V_J=\mathbb R\omega_0,
\]
and $\rho^B_{\omega_0}=0.$

Consequently, the solution of the constrained Bismut--Ricci flow is stationary:
\[
\omega(t)\equiv\omega_0,
\qquad -\infty<t<\infty.
\]
In particular, every $J$-compatible almost-SKT metric has fundamental
form $\lambda\omega_0$ for some constant $\lambda>0$ and has vanishing
Bismut Ricci form.

\end{cor}

\begin{proof}
By Theorem \ref{thr3.16-dimsix}, non-integrability implies that
$\widetilde N_{\omega_0}$ is nowhere vanishing and that
$\rho^B_{\omega_0}=0.$

Proposition~\ref{lem:dimension-six-rigidity-final} therefore gives
pointwise Nijenhuis rigidity. Proposition~\ref{prop:rigid-kernel-final} then yields
$\mathcal V_J=\mathbb R\omega_0,$
while Theorem~\ref{thm:global-rigid-moduli} gives
\[
\mathcal A_J=\{\lambda\omega_0:\lambda>0\}.
\]
Theorem~\ref{thm:rigid-complete-dynamics} then applies. Since
$\rho^B_{\omega_0}=0$, the constant $\mu_0$ defined in \eqref{eq:rigid-mu-all-dim}  vanishes.
Therefore the solution \eqref{eq:rigid-homothetic-all-dim}  reduces to $\omega(t)=\omega_0$
for all $t\in\mathbb R$.
\end{proof}

 \begin{remark}
\label{rem:scope-OE}
In real dimension six, Theorem~\ref{thr3.16-dimsix} shows that there is no
intermediate non-integrable almost-SKT regime in which the skew
Nijenhuis form vanishes on a proper subset. Indeed, if $J$ is
non-integrable, then $\widetilde N_{\omega_0}$ is nowhere vanishing.
Hence Proposition~\ref{lem:dimension-six-rigidity-final} gives
pointwise Nijenhuis rigidity, and therefore \emph{(OE)} holds
everywhere.

Thus, on a compact non-integrable almost-SKT six-manifold,
Proposition~\ref{thm:Bismut-short-time-final} applies, while
Corollary  \ref{cor3.17-dimsix}  shows in fact that the constrained Bismut--Ricci
flow solution is stationary.

In higher dimensions, by contrast, \emph{(OE)} may fail even when the
skew Nijenhuis tensor is nonzero. In that case the admissible space
$\mathscr V_J$ may be infinite dimensional and the finite-dimensional
argument above no longer applies. A parabolic or Hodge-type analysis of
the constrained problem would then be required, as in the pluriclosed
case of Remark~\ref{rem:pluriclosed-limit}.
\end{remark}
\begin{remark}\label{ex:S3S3-strong-ABHE}
In \cite[Example~5.8]{FinoFowdar2026}, a left-invariant $SU(3)$-structure
on $S^3\times S^3$ is constructed for which the Nijenhuis tensor is a
nonzero three-form, the characteristic torsion is closed, and in fact the
characteristic connection is flat.  Hence
\[
\widetilde N\neq0,\qquad dT^B=0,\qquad \rho^B=0.
\]
The structure is therefore a concrete non-integrable example satisfying
all three conditions above.  Since it is left invariant and $\widetilde N$
is nonzero, the Nijenhuis form is nowhere vanishing, so
Corollary~\ref{cor3.17-dimsix} applies.  The
constrained Bismut--Ricci flow is stationary because $\rho^B=0$.
\end{remark}

\begin{remark}
\label{rem:Schoemann-parallel-strong-flat}
There is a structural reason why several natural homogeneous six-dimensional
examples with nonzero skew Nijenhuis tensor and closed characteristic torsion turn
out to be Bismut-flat.  Schoemann \cite{Schoemann2007} studies almost Hermitian six-manifolds with
parallel characteristic torsion.  In the class of type \(\mathcal W_1\oplus \mathcal W_3\) related
to \(SO(3)\), the characteristic curvature and the exterior derivative of the
torsion are controlled by the same scalar
\[
\lambda=\|T_2\|^2-\frac13\|T_{12}\|^2:
\qquad
R^B=-\lambda\,\operatorname{pr}_{\mathfrak{so}(3)},
\qquad
dT^B=\lambda\,*\omega,
\]
see \cite[Lemma~4.1 and Proposition~4.6]{Schoemann2007}.  Consequently, for this class of structures of type $\mathcal W_1\oplus \mathcal W_3$ with parallel characteristic torsion and $SO(3)$-holonomy,
\[
dT^B=0\quad\Longleftrightarrow\quad R^B=0.
\]
\begin{sloppypar}
In particular, the naturally reductive \(S^3\times S^3\) families of
\cite[Example~4.10]{Schoemann2007} cannot provide a strong non-integrable
example with non-flat characteristic connection.  This motivates looking beyond
the parallel-torsion $SO(3)$-holonomy setting.
\end{sloppypar}
\end{remark}

\section{Almost Bismut--Hermitian--Einstein structures and generalized Ricci flow}

This section relates almost-SKT geometry to the  generalized Ricci flow.  We
introduce the almost Bismut Hermitian--Einstein (ABHE) condition 
and study its relation to generalized Ricci solitons and fixed points.

  We recall that, with the convention used here, the generalized Ricci flow on a manifold $M$ is a one-parameter family of Riemannian metrics $g(t)$ and closed three-forms $H(t)\in\Omega^3(M)$ satisfying \begin{equation}\label{eq:generalized-Ricci-flow} \partial_t g = -2\left(\operatorname{Ric}^g-\frac14 H^2\right), \qquad \partial_t H = -dd_g^*H, \end{equation} where $ H^2(X,Y) = \langle\iota_XH,\iota_YH\rangle_g$,  see \cite{GarciaFernandezStreets2020}.

\subsection{Almost Bismut Hermitian--Einstein structures and generalized Ricci flow} \label{sub4.1}

We first distinguish the first Bismut Ricci form from the Ricci tensor of the
characteristic connection. Throughout this subsection,
$\operatorname{Ric}^B$ denotes the Ricci tensor of $\nabla^B$, whereas
$\rho^B$ denotes its first Ricci form, as defined in Section~2.
We put \(T=T^B\) and use the Lee-form convention of
\cite{IvanovStanchev2023}: if \(F(X,Y)=g(X,JY)=\omega(X,Y)\), then
\(\theta=\delta F\circ J\). 

\begin{definition}
\begin{sloppypar}
An almost Hermitian structure \((J,\omega)\) is called
\emph{strong almost Bismut Hermitian--Einstein} (strong ABHE) if
\end{sloppypar}
\begin{equation}\label{eq:strong-ABHE}
\widetilde N_\omega\in\Omega^3(M),\qquad
 dT^B_\omega=0,\qquad
 \rho^B_\omega=0.
\end{equation}
\end{definition}

\begin{remark}  In real dimension six, an almost Hermitian structure with skew Nijenhuis
tensor and $\rho^B=0$ is called \emph{almost-CYT} (ACYT) in
\cite{FinoFowdar2026}; strong ABHE adds the almost-SKT equation
$dT^B=0$.
\end{remark}

 For later use, we recall the corresponding identities in our conventions. The first is the identity underlying
\cite[Theorem~3.5(1)]{KennonStreets}. Under the strong ABHE conditions
\[
dT=0,\qquad \rho^B=0,
\]
one has
\begin{equation}\label{eq:RicB-Lee}
\operatorname{Ric}^B=-\nabla^B\theta
\end{equation}
and
\begin{equation}\label{eq:deltaT-Lee}
\delta T=d\theta-\iota_{\theta^\sharp}T.
\end{equation}
Indeed, the characteristic connection satisfies
\begin{equation}\label{eq:rhoB-RicB}
\rho^B(X,Y)
=
\operatorname{Ric}^B(X,JY)
+(\nabla^B_X\theta)(JY)
+\frac14 dT(X,Y,e_a,Je_a),
\end{equation}
see \cite[Eq.~(2.19)]{IvanovStanchev2023}. Hence, since $dT=0$ and
$\rho^B=0$, replacing $Y$ by $-JY$ gives
\eqref{eq:RicB-Lee}.

The second identity follows by taking the skew-symmetric part of
\eqref{eq:RicB-Lee}. Indeed,
\[
\operatorname{Ric}^B_{ij}-\operatorname{Ric}^B_{ji}
=-(\delta T)_{ij},
\]
while \eqref{eq:RicB-Lee} gives
\[
\operatorname{Ric}^B_{ij}-\operatorname{Ric}^B_{ji}
=
-\nabla^B_i\theta_j+\nabla^B_j\theta_i.
\]
Therefore
\[
(\delta T)_{ij}
=
\nabla^B_i\theta_j-\nabla^B_j\theta_i.
\]
On the other hand,
\[
(d\theta)_{ij}
=
\nabla^B_i\theta_j-\nabla^B_j\theta_i
+\theta_sT_{sij}.
\]
It follows that
\[
(\delta T)_{ij}
=
(d\theta)_{ij}-\theta_sT_{sij},
\]
which is precisely \eqref{eq:deltaT-Lee}.

We shall also use the standard relation
\begin{equation}\label{eq:Ricci-comparison}
\operatorname{Ric}^g
=
\operatorname{Ric}^B+\frac12\delta T+\frac14T^2.
\end{equation}

 The identities \eqref{eq:RicB-Lee} and \eqref{eq:deltaT-Lee}  are local and require
only $dT=0$ and $\rho^B=0$. By \cite[Remark~3.4]{KennonStreets}, the condition
$\rho^B=0$ is equivalent to reduced holonomy.  

\subsection{Gradient solitons and fixed points}
 We first recall the compact gradient-soliton statement and then derive
two fixed-point consequences.

 In the compact case, \cite[Theorem~3.5]{KennonStreets} shows that
there exists a smooth function $f$ such that $(g,T,f)$ is a steady
generalized gradient Ricci soliton,
\[
\operatorname{Ric}^g-\frac14T^2+\nabla^{g,2}f=0,
\qquad
\delta T+\iota_{\nabla f}T=0,
\qquad
dT=0,
\]
and the vector field $V$ defined by
\[
V^\flat=\theta-df
\]
is $\nabla^B$-parallel. In particular, $V$ is Killing.
In contrast with the integrable BHE case, a characteristic-parallel field does
not automatically preserve the almost complex structure.  The obstruction is
completely explicit.

\begin{prop}
\label{prop:fixed-J-gauge-obstruction}
If \(V\) is \(\nabla^B\)-parallel, then
\begin{equation}\label{eq:LVJ-obstruction}
(\mathcal L_VJ)X
=
J\!\left(T(V,X,\cdot)^\sharp\right)
-
T(V,JX,\cdot)^\sharp.
\end{equation}
Consequently, the local flow generated by $V$ preserves $J$
if and only if
\begin{equation}\label{eq:torsion-J-V-condition}
T(V,JX,\cdot)^\sharp
=J\!\left(T(V,X,\cdot)^\sharp\right)
\qquad\text{for every }X.
\end{equation}
\end{prop}

\begin{proof}
Using \(\nabla^BJ=0\), \(\nabla^BV=0\), and
\([X,Y]=\nabla^B_XY-\nabla^B_YX-T(X,Y,\cdot)^\sharp\), one computes
\[
(\mathcal L_VJ)X=[V,JX]-J[V,X]
=J(T(V,X,\cdot)^\sharp)-T(V,JX,\cdot)^\sharp.
\]
\end{proof}
 This formula describes the obstruction to a
characteristic-parallel vector field preserving the almost
complex structure.
  In the compact six-dimensional ACYT case with closed
characteristic torsion, the parallel field
$V=\theta^\sharp-\nabla f$ preserves
the
full \(SU(3)\)-structure; see \cite[Theorem~6.5]{IvanovStanchev2023}.

There are two useful fixed-point consequences.

\begin{cor}
\label{cor:balanced-ABHE-GRF-fixed}
Let \((M^{2n},g,J)\) be strong ABHE.  If \(\theta=0\), then
\begin{equation}\label{eq:balanced-ABHE-fixed-identities}
\operatorname{Ric}^B=0,
\qquad
\delta T=0,
\qquad
 dT=0.
\end{equation}
Hence
\begin{equation}\label{eq:generalized-Ricci-flat-pair}
\operatorname{Ric}^g-\frac14T^2=0,
\qquad
\Delta_dT=0,
\end{equation}
and \((g,T)\) is a fixed point of generalized Ricci flow.
\end{cor}

\begin{proof}
 
Equations \eqref{eq:RicB-Lee} and  \eqref{eq:deltaT-Lee} immediately give \eqref{eq:balanced-ABHE-fixed-identities}.
The conclusion follows from \eqref{eq:Ricci-comparison} and $dT=\delta T=0$.

\end{proof}

{\begin{remark}
The example of \cite[Example~5.8]{FinoFowdar2026} on $S^3\times S^3$
discussed in Remark~\ref{ex:S3S3-strong-ABHE} is balanced in this sense:
the Lee component of its $SU(3)$ torsion vanishes, so $\theta=0$.
Corollary~\ref{cor:balanced-ABHE-GRF-fixed} therefore identifies the pair
$(g,T^B)$ for that example as an actual fixed point of generalized Ricci
flow, not merely a stationary point of the constrained Bismut--Ricci flow.
\end{remark}
 On a compact manifold, the balanced assumption in
Corollary  \ref{cor:balanced-ABHE-GRF-fixed} can be weakened to the co-closedness
of the Lee form. In the integrable case, the equivalence
between a constant soliton potential and the Gauduchon
condition was established in
\cite[Lemma~2.9]{ABLS}. The following argument extends
this implication to strong ABHE structures.

\begin{cor}
\label{cor:coclosed-Lee-ABHE-fixed}
\begin{sloppypar}
Let \((M^{2n},g,J)\) be compact and strong ABHE.  If \(\delta\theta=0\), then
 the soliton potential $f$ given by
\cite[Theorem~3.5]{KennonStreets} is constant,
\(\nabla^B\theta=0\), and \((g,T)\) is a generalized Ricci fixed point.
\end{sloppypar}
\end{cor}

\begin{proof}
By \cite[Theorem~3.5]{KennonStreets}, write
\[
\theta=df+V^\flat,
\quad
\nabla^B V=0.
\]
Since $V$ is Killing, $\delta V^\flat=0$. Thus $\delta\theta=0$
gives $\delta df=0$, and compactness implies that $f$ is constant.
Therefore $\theta=V^\flat$ is $\nabla^B$-parallel. Equation~(31)
gives $\operatorname{Ric}^B=0$, and the identity
\[
\operatorname{Ric}^B_{ij}-\operatorname{Ric}^B_{ji}
=-(\delta T)_{ij}
\]
then gives $\delta T=0$. Hence the previous fixed-point argument applies.
\end{proof}

The vanishing Einstein constant is not merely a normalization in the genuinely
non-integrable case.

\begin{prop}
\label{prop:nonzero-Bismut-Einstein-Kahler}
Let \((M^{2n},g,J)\) admit a characteristic connection and suppose that, for a
constant \(\lambda\neq0\),
\begin{equation}\label{eq:lambda-ABHE}
\rho^B=\lambda\omega.
\end{equation}
Then \(d\omega=0\), \(N_J=0\), and \((g,J)\) is K\"ahler.  Consequently
\eqref{eq:lambda-ABHE} is the ordinary K\"ahler--Einstein equation.  In
particular, a  non-integrable Bismut Hermitian--Einstein analogue
with constant Einstein factor must have \(\lambda=0\).
\end{prop}

\begin{proof}
The first Ricci form of a unitary connection is closed by the Bianchi identity,
so \(0=d\rho^B=\lambda d\omega\), and \(d\omega=0\).  On an almost Hermitian
manifold admitting a characteristic connection, the \((3,0)+(0,3)\) component
of \(d\omega\) determines the skew Nijenhuis tensor; explicitly, up to the
fixed convention,
\[
(d\omega)^-(X,Y,Z)=\frac34\widetilde N(JX,Y,Z),
\]
see \cite[Section~2.2]{IvanovStanchev2023}.  Hence \(N_J=0\).  Thus \(J\) is
integrable and \(d\omega=0\), so the structure is K\"ahler.  Its characteristic
connection is the Levi--Civita connection and \eqref{eq:lambda-ABHE} is the
K\"ahler--Einstein equation.
\end{proof}

\begin{remark}
The weaker equation \((\rho^B)^{1,1}=\lambda\omega\) does not yield the same
immediate closedness argument and is a distinct rigidity problem.  Even in the
integrable theory it is substantially weaker than
\(\rho^B=\lambda\omega\); we do not impose it as the definition of ABHE here.
\end{remark}

\section{Examples and constructions}
{ We first rule out a natural attempt at a non-flat non-integrable example on
$SU(3)/T^2$; give a local cohomogeneity-one construction on $I\times S^5$
where this obstruction is overcome; and then, via a quaternionic rotation,
produce compact examples on $S^1\times S^2\times S^3$ and $S^3\times S^3$.}
\subsection{A no-go calculation on the flag manifold $SU(3)/T^2$}
{ Having introduced the ABHE condition, we now turn to constructing
 non-flat non-integrable examples.  The flat example of
Remark~\ref{ex:S3S3-strong-ABHE} suggests asking whether Bismut-flatness
is merely a consequence of its very special naturally reductive
presentation. The next computation shows that a larger
and natural homogeneous ansatz still does not produce an almost-SKT
structure compatible with the given non-integrable almost complex structure.
}

\begin{prop}
\label{prop:flag-no-strong}
Let $F_{1,2}=\mathrm{SU}(3)/T^2$ be the full flag manifold.  Write the isotropy
representation as the sum of its three real two-dimensional root modules
$\mathfrak m_1\oplus\mathfrak m_2\oplus\mathfrak m_3$ and consider the general
invariant metric
\[
 g=x_1Q|_{\mathfrak m_1}+x_2Q|_{\mathfrak m_2}+x_3Q|_{\mathfrak m_3},
 \qquad x_i>0,
\]
where $Q=-\frac12\operatorname{tr}(XY)$.  Choose a $g$-orthonormal frame
$e_1,\ldots,e_6$ adapted to the three root modules and the non-integrable
invariant almost complex structure
\[
 Je_1=e_2,\qquad Je_3=e_4,\qquad Je_5=-e_6.
\]
Then $\widetilde N$ is a three-form if and only if
\[
 x_1=x_2=x_3.
\]
For $x_1=x_2=x_3=s^2$, the characteristic torsion satisfies
\[
 T^B=-\frac1s\bigl(e^{135}+e^{146}+e^{236}-e^{245}\bigr),
\]
and
\[
 dT^B=\frac4{s^2}\bigl(e^{1234}-e^{1256}-e^{3456}\bigr)\neq0.
\]
Consequently no invariant metric compatible with this non-integrable invariant
almost complex structure is almost-SKT.
\end{prop}

\begin{proof} 
Use the standard reductive complement consisting of the off-diagonal
skew-Hermitian matrices. If $r_i=\sqrt{x_i}$, the three failures of total
skew-symmetry of the lowered Nijenhuis tensor are, up to a common nonzero
factor,
\[
r_2^2-r_3^2,\qquad r_1^2-r_3^2,\qquad r_1^2-r_2^2.
\]
{ For instance, evaluating the skew-symmetry defect on a triple with one
vector in each root module gives, up to the same nonzero factor,
\(r_2^2-r_3^2\); the other two expressions follow by cyclic permutation.}
Hence the $G_1$ condition holds if and only if  $r_1=r_2=r_3=s$. Substitution
into the characteristic-torsion formula gives the displayed $T^B$. Applying
the homogeneous Chevalley--Eilenberg differential yields the displayed
nonzero four-form. The invariant almost complex structures on this
homogeneous space and their Gray--Hervella types are discussed more generally
in \cite{AlekseevskyKruglikovWinther2014}.
\end{proof}

\begin{remark}
The proposition above points in the same direction as
Remark~\ref{rem:Schoemann-parallel-strong-flat}: 
 an example with
\[
 N_J\neq0,\qquad dT^B=0,\qquad R^B\neq0
\]
 must be found beyond the homogeneous constructions considered above.
Proposition~\ref{prop:S5-cohom1-strong-ABHE}
below shows that breaking homogeneity by a cohomogeneity-one deformation is
already sufficient locally.
\end{remark}

\subsection{A cohomogeneity-one example on $S^5$}
The following construction is local and does not require compactness. It provides a strong ABHE structure whose characteristic
connection is not flat.

\begin{prop}
\label{prop:S5-cohom1-strong-ABHE}
Let \((\eta,\omega_1,\omega_2,\omega_3)\) be the standard Sasaki--Einstein
$SU(2)$-structure on $S^5$, normalized by
\begin{equation}\label{eq:S5-SU2-structure}
 d\eta=-2\omega_3,\qquad
 d\omega_1=3\eta\wedge\omega_2,\qquad
 d\omega_2=-3\eta\wedge\omega_1,\qquad
 d\omega_3=0.
\end{equation}
These are the Einstein--Sasaki equations of \cite[Proposition~11]{ContiSalamon2007}
after a cyclic relabelling of the three two-forms.  Let $r$ denote the
coordinate on an interval $I$, and let $x,b>0$ and $\vartheta$ be
functions of $r$.  Put $s=\sin\vartheta$, $c=\cos\vartheta$ and
\begin{equation}\label{eq:S5-ABC}
 A=c\,\omega_1+s\,\omega_3,\qquad
 B=\omega_2,\qquad
 C=-s\,\omega_1+c\,\omega_3.
\end{equation}
Then, on every open interval on which \(sc\neq0\), the system
\begin{equation}\label{eq:S5-strong-ODE}
\begin{aligned}
 \vartheta'&=\frac{s(4b^2-3x)}{bx},\\
 b'&=\frac{3x-2b^2}{xc},\\
 x'&=\frac{2b^2(c^2-s^2)+3xs^2}{bc}
\end{aligned}
\end{equation}
defines a local almost Hermitian structure on \(I\times S^5\) satisfying
\begin{equation}\label{eq:S5-strong-properties}
 \widetilde N\in\Omega^3,\qquad
 \widetilde N\neq0,\qquad
 dT^B=0,\qquad
 \rho^B=0.
\end{equation}
Moreover, such structures need not be Bismut-flat.  For example, the solution
with
\begin{equation}\label{eq:S5-initial-data}
 x(0)=b(0)=1,\qquad \vartheta(0)=\frac\pi4
\end{equation}
has \(\operatorname{Ric}^B\neq0\) at \(r=0\), and hence \(R^B\neq0\).
\end{prop}

\begin{proof}  We retain throughout the convention
$\omega(X,Y)=g(X,JY)$ fixed in Section~2.
Consider first the full geometric $SO(3)$-rotation of the Sasaki--Einstein
triple.  Up to the phase of the complex volume form, it is parametrized by a
unit vector \(C:I\to S^2\).  Locally write
\begin{align}\label{eq:S5-full-SO3}
 A&=\cos\vartheta(\cos\varphi\,\omega_1+\sin\varphi\,\omega_2)
       +\sin\vartheta\,\omega_3,\notag\\
 B&=-\sin\varphi\,\omega_1+\cos\varphi\,\omega_2,\\
 C&=-\sin\vartheta(\cos\varphi\,\omega_1+\sin\varphi\,\omega_2)
       +\cos\vartheta\,\omega_3.\notag
\end{align}
The remaining Euler angle rotates \(A+iB\) by a phase and therefore changes
neither \((g,J, \omega)\) nor the tensors considered below.  Define
\begin{equation}\label{eq:S5-SU3-ansatz}
 \omega=-xC-b\,\eta\wedge dr,
\qquad
\Psi=\Psi_++i\Psi_-
=x(A+iB)\wedge(b\eta+i\,dr),
\end{equation}
so that
\begin{equation}\label{eq:S5-metric}
 g=xg_H+b^2\eta^2+dr^2,
\end{equation}
where \(g_H\) is the horizontal Sasaki--Einstein metric. In particular,
\[
\Psi_+
=x\bigl(b\,\eta\wedge A-dr\wedge B\bigr),
\qquad
\Psi_-
=x\bigl(dr\wedge A+b\,\eta\wedge B\bigr).
\]
 Using \eqref{eq:S5-SU2-structure} and the \(SU(3)\)-torsion identities
written with our convention \(\omega(X,Y)=g(X,JY)\), the condition that
the Nijenhuis tensor be totally skew, equivalently
\(\pi_2=\sigma_2=0\), gives
on the locus \(s\neq0\)
\begin{equation}\label{eq:S5-skew-N-system}
 \varphi'=0,\qquad
 \vartheta'=\frac{s(4b^2-3x)}{bx},\qquad
 \pi_0=-\frac{2bs}{x},\qquad \sigma_0=0,
\end{equation}
and
\begin{equation}\label{eq:S5-pi1}
 \pi_1=
 \left(\frac{x'}x+\frac{b'}b-\frac{3c}{b}\right)dr.
\end{equation}
The constant value of \(\varphi\) can be absorbed into the fixed
\(U(1)\)-rotation of \((\omega_1,\omega_2)\), so we henceforth set
\(\varphi=0\), obtaining \eqref{eq:S5-ABC}. From \eqref{eq:S5-SU2-structure} and \eqref{eq:S5-ABC} we obtain
\begin{align}
dA&=3c\,\eta\wedge B+\vartheta'\,dr\wedge C,\\
dB&=-3c\,\eta\wedge A+3s\,\eta\wedge C,\\
dC&=-\vartheta'\,dr\wedge A-3s\,\eta\wedge B,
\end{align}
while $d\eta=-2(sA+cC).$ By the formula for the skew Nijenhuis form in
\cite[Proposition~4.1]{FinoFowdar2026}, after translating to our
convention for \(N_J\), we obtain
\begin{equation}\label{eq:S5-N-form} \widetilde N
=-4bs\bigl(b\,\eta\wedge A-dr\wedge B\bigr)
=-\frac{4bs}{x}\Psi_+.
\end{equation}
In particular \(\widetilde N\neq0\) precisely when \(s\neq0\).

The characteristic torsion is
\begin{equation}\label{eq:S5-T-PQR}
 T^B=P\,\eta\wedge C+Q\,\eta\wedge A+R\,dr\wedge B,
\end{equation}
where
\begin{equation}\label{eq:S5-PQR}
 P=b(x'-2bc),\qquad
 Q=s(2b^2-3x),\qquad
 R=\frac{s(4b^2-3x)}b=x\vartheta'.
\end{equation}
 Let $v_H$ be the horizontal volume form normalized by
\(\omega_i^2=2v_H\). Since
\[
A^2=B^2=C^2=2v_H,
\qquad
A\wedge B=A\wedge C=B\wedge C=0,
\]
we have
\begin{align}
d(\eta\wedge C)
&=-4c\,v_H+\vartheta'\eta\wedge dr\wedge A,\\
d(\eta\wedge A)
&=-4s\,v_H-\vartheta'\eta\wedge dr\wedge C,\\
d(dr\wedge B)
&=-3c\,\eta\wedge dr\wedge A
  +3s\,\eta\wedge dr\wedge C.
\end{align}
Therefore, direct differentiation gives
\begin{align}\label{eq:S5-dT-decomp}
 dT^B={}&-4(cP+sQ)v_H
 +(P\vartheta'-Q'-3cR)\,\eta\wedge dr\wedge A\notag\\
 &+(-P'-Q\vartheta'+3sR)\,\eta\wedge dr\wedge C.
\end{align}
Combining \eqref{eq:S5-skew-N-system} with \(dT^B=0\), the first two
independent components of \eqref{eq:S5-dT-decomp} give the last two equations
of \eqref{eq:S5-strong-ODE}; after differentiating the equation for \(x'\)
and using the equations for \(b'\) and \(\vartheta'\), the remaining component
vanishes identically.  Hence \eqref{eq:S5-strong-ODE} is exactly the strong
system on \(sc\neq0\).

The \(SU(3)\) Lee component is
\begin{equation}\label{eq:S5-nu1}
 \nu_1=\lambda\,dr,\qquad
 \lambda=\frac{x'-2bc}{x}
 =\frac{(3x-2b^2)s^2}{bcx},
\end{equation}
and substitution of \eqref{eq:S5-strong-ODE} into \eqref{eq:S5-pi1} yields
\begin{equation}\label{eq:S5-pi1-2nu1}
 \pi_1=2\nu_1.
\end{equation}
It also gives \(\pi_0'=0\).  Corollary~4.8 of \cite{FinoFowdar2026}
(or the corresponding \(SU(3)\)-torsion formula for the first Bismut Ricci
form) therefore gives \(\rho^B=0\).  This proves
\eqref{eq:S5-strong-properties}.

It remains to see that the characteristic connection need not be flat.  In the
Lee-form convention of this paper one has \(\theta=2\nu_1=2\lambda\,dr\).
 Equation~\eqref{eq:RicB-Lee} therefore implies
\begin{equation}\label{eq:S5-SymRicB}
 \operatorname{Sym}\operatorname{Ric}^B
 =-\frac12\mathcal L_{\theta^\sharp}g
 =-\mathcal L_{\lambda\partial_r}g
 =-\lambda x' g_H-2\lambda bb'\eta^2-2\lambda' dr^2.
\end{equation}
For the initial data \eqref{eq:S5-initial-data}, the system gives
\begin{equation}\label{eq:S5-initial-derivatives}
\begin{aligned}
 \vartheta'(0)&=\frac1{\sqrt2}, &
 b'(0)&=\sqrt2, &
 x'(0)&=\frac3{\sqrt2},\\
 \lambda(0)&=\frac1{\sqrt2}, &
 \lambda'(0)&=-\frac12.
\end{aligned}
\end{equation}
Consequently
\begin{equation}\label{eq:S5-initial-SymRic}
 \left.\operatorname{Sym}\operatorname{Ric}^B\right|_{r=0}
 =-\frac32g_H-2\eta^2+dr^2\neq0.
\end{equation}
Thus \(\operatorname{Ric}^B\neq0\), and in particular \(R^B\neq0\), at the
initial orbit.  Standard ODE existence gives a local solution through every
initial datum with \(x,b>0\) and \(sc\neq0\).
\end{proof}
{Notice that, along any solution of the system \eqref{eq:S5-strong-ODE} on which $sc\neq0$, the quantity $\mu=\frac{bs}{x}$ is constant, since} $$\frac{\mu'}
{\mu}=\frac{b'}{b}+\frac{c}{s}\vartheta'-\frac{x'}{x}=0.$$  Setting $y=\frac{x}{b^2}$  and using $\mu=\frac{s}{yb}$, the system \eqref{eq:S5-strong-ODE} gives \begin{equation}
    \begin{split}
        \vartheta'=&\mu(4-3y)\\
        b'=&\frac{3y-2}{yc}
    \end{split}
\end{equation}
Now, using that $y=\frac{x}{b^2}$, we get 
\begin{equation}
    y'=\frac{(2-3y)+(4-3y)c^2}{bc}
\end{equation} Since $x,b>0$, we also have  $y>0$. { On the region
\[
b>0,\qquad 0<\vartheta<\frac{\pi}{2},
\]
introduce a new parameter $\tau$ by
$\frac{dr}{d\tau}=bc.$ Then  the system becomes} \begin{align}
          \frac{d\vartheta}{d\tau}=&\frac{cs}{y}(4-3y)\label{eq.1}\\
        \frac{dy}{d\tau}=&(2-3y)+(4-3y)c^2\label{eq.2}\\
        \frac{db}{d\tau}=&\frac{b}{y}(3y-2)   \label{eq.3}
\end{align}
The first two equations form a closed autonomous system in the plane.
On the region
\[
0\leq\vartheta\leq\frac{\pi}{2},
\qquad
\frac23\leq y\leq1,
\]
we have $4-3y>0$, so equation \eqref{eq.1} vanishes only when
$s=0$ or $c=0$. If $s=0$, then $\vartheta=0$ and $c^2=1$.
Substituting $c^2=1$ into \eqref{eq.2}, we obtain $y=1$, and hence
$p_0=(0,1).$
If $c=0$, then $\vartheta=\pi/2$, and equation \eqref{eq.2} gives
$y=2/3$, yielding the second equilibrium
$p_1=\left(\frac{\pi}{2},\frac23\right).$
Thus $p_0$ and $p_1$ are boundary equilibria of the rescaled planar
system.

{ Let $P$ denote the vector field defined by the right-hand sides of
\eqref{eq.1} and \eqref{eq.2}. Its Jacobian matrices at the two boundary
equilibria are
\begin{equation}
    DP(p_0)=
    \begin{pmatrix}
        1&0\\
        0&-6
    \end{pmatrix},
    \qquad
    DP(p_1)=
    \begin{pmatrix}
        -3&0\\
        0&-3
    \end{pmatrix}.
\end{equation}
Therefore, $p_0$ is a saddle point, whereas $p_1$ is an attracting
hyperbolic equilibrium of the rescaled planar system.

Consider the region
\[
\mathcal R=
\left\{
0<\vartheta<\frac{\pi}{2},
\quad
\frac{2}{3}<y<1
\right\}.
\]
On the boundary $y=\frac{2}{3}$, equation \eqref{eq.2} gives
$\dot y=2c^2>0$, whereas on $y=1$ it gives $\dot y=-s^2<0$.
Furthermore, throughout $\mathcal R$ one has $s>0$, $c>0$, and
$4-3y>1$, and therefore $\dot\vartheta>0$. Thus $\mathcal R$ is
forward invariant. Since $\vartheta$ is strictly increasing and cannot
converge to a value smaller than $\frac{\pi}{2}$, equation
\eqref{eq.2} implies that every trajectory in $\mathcal R$ converges
to $p_1$ as $\tau\to+\infty$. In particular, the trajectory shown in
Figure~\ref{fig:mi_imagen}, which leaves $p_0$ and enters
$\mathcal R$, converges to $p_1$.

\begin{figure}[htbp]
  \centering
  \includegraphics[width=0.8\textwidth]{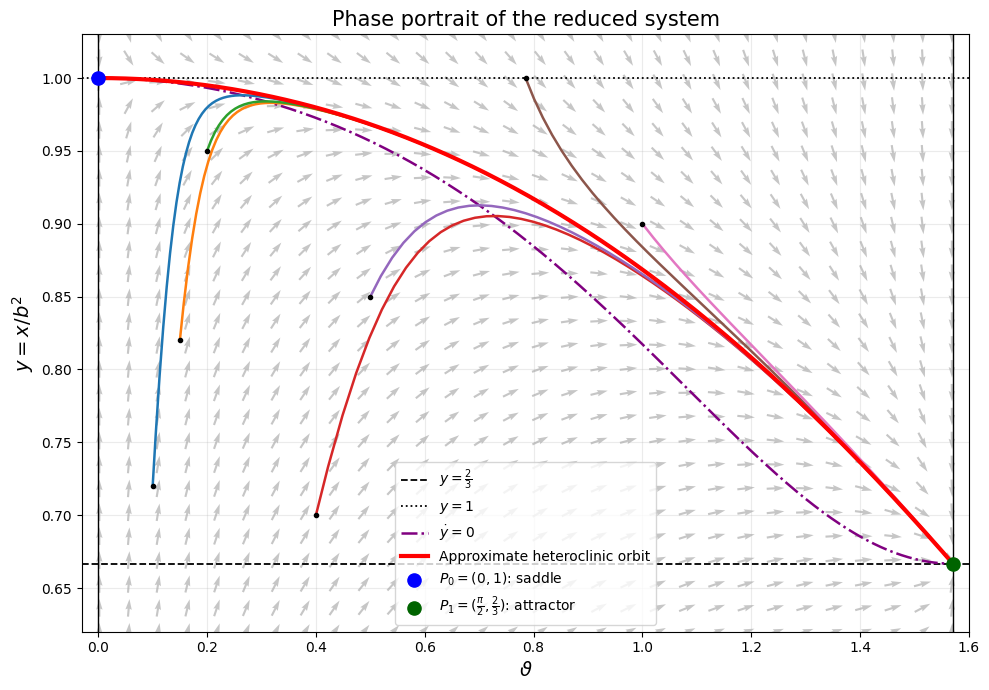}
  \caption{The autonomous system \eqref{eq.1}--\eqref{eq.2}.}
  \label{fig:mi_imagen}
\end{figure}

The first boundary equilibrium $p_0=(0,1)$ has the following geometric
interpretation. Substitution into the original system
\eqref{eq:S5-strong-ODE} gives
\[
\vartheta'=0,\qquad b'=1,\qquad x'=2b.
\]
Thus
\[
b(r)=r-r_0,\qquad x(r)=(r-r_0)^2.
\]
Consequently, the skew Nijenhuis form satisfies
\[
\widetilde N
=-4bs\bigl(b\,\eta\wedge A-dr\wedge B\bigr)=0,
\]
and hence $N_J=0$. Moreover,
\[ \omega=-b^2\omega_3-b\,\eta\wedge dr
\]
is closed, so $p_0$ corresponds to the K\"ahler cone over $S^5$.

The point $p_1=(\pi/2,2/3)$ is a boundary equilibrium of the rescaled
system, rather than an equilibrium of the original system
\eqref{eq:S5-strong-ODE}, whose right-hand side is singular when $c=0$.
Nevertheless, it describes the limiting configuration of trajectories
in $\mathcal R$. Indeed, along every such trajectory,
\[
b\longrightarrow\frac{3}{2\mu},
\qquad
x\longrightarrow\frac{3}{2\mu^2}.
\]
Moreover, equations \eqref{eq:S5-SU3-ansatz} and \eqref{eq:S5-N-form} give $\widetilde N=-4\mu\Psi_+.$

Since $\mu>0$ and $\Psi_+$ is a stable form, it follows that
$\widetilde N\neq0$.}

\begin{remark}
The computation above starts with the full geometric \(SO(3)\)-rotation of the
horizontal \(SU(2)\)-triple, not merely the one-angle subfamily. On the
non-integrable skew-Nijenhuis locus, \eqref{eq:S5-skew-N-system} forces the
azimuthal angle \(\varphi\) to be constant, while the third Euler angle is only
the phase of \(\Psi\). Thus, modulo the fixed \(U(1)\)-symmetry, the full
ansatz reduces to \((x,b,\vartheta)\).

Moreover, \eqref{eq:S5-pi1-2nu1} shows directly that the strong locus in this
ansatz is almost-CYT. This is consistent with Theorem~\ref{thr3.16-dimsix}, which shows that
every non-integrable almost-SKT six-manifold has vanishing Bismut Ricci form.
Hence the corresponding  unprojected Bismut--Ricci flow is necessarily
stationary, despite the fact that \(R^B\neq0\).
\end{remark}

\subsection{Quaternionic rotations and compact examples}
 Recall that a Sasaki-with-torsion manifold is a normal almost
contact metric manifold $(M,\phi,\xi,\eta,g)$ whose Reeb vector
field $\xi$ is Killing. Such a structure admits a unique metric
connection with totally skew-symmetric torsion preserving
$\phi$, $\xi$, and $\eta$, called its characteristic connection.
The structure is called strong when its characteristic torsion
is closed \cite{BrienzaFinoFowdarGrantcharov2026}.

\begin{thm}
\label{thm:ST-quaternionic-rotation}
Let $(M^5,\varphi,\xi,\eta,g,\nabla)$ be a compact simply connected strong
Sasaki-with-torsion manifold whose characteristic torsion is $H$.

Denote by
$F_1$ its fundamental $2$-form,
$F_1(X,Y)=g(X,\varphi Y),$
and by $\rho^\nabla$ the Ricci form of the characteristic connection
$\nabla$. Assume
that
\begin{equation}\label{eq:ST-rotation-hyp}
 \rho^\nabla=0,
 \quad
 \mathfrak c:=\frac12\langle d\eta,F_1\rangle_g\neq0, \,
 \quad\text{everywhere on }M.
\end{equation}
Then \(X=S^1_s\times M\), equipped with \(G=ds^2+g\), carries an
\(S^2\)-family \(\{J_q\}_{q\in S^2}\) of orthogonal almost complex structures
such that
\begin{equation}\label{eq:ST-rotation-properties}
 \widetilde N_{J_q}\in\Omega^3(X),
 \qquad
 T^B_{J_q}=H,
 \qquad
 dT^B_{J_q}=0,
 \qquad
 \rho^B_{J_q}=0.
\end{equation}
There are two antipodal parameters \(\pm q_0\) for which \(J_q\) is
integrable.  For every
\(q\in S^2\setminus\{\pm q_0\}\),
\begin{equation}\label{eq:ST-rotation-nonintegrable}
N_{J_q}(p)\neq0, 
\,
\forall p\in X.
\end{equation}
Moreover the characteristic curvature satisfies
\begin{equation}\label{eq:ST-rotation-curvature}
 R^B_{J_q}=R^\nabla.
\end{equation}
In particular, if \(\nabla\) is non-flat, then every non-integrable member of
the family is a strong ABHE structure which is not 
Bismut-flat.
\end{thm}

\begin{proof}
Put \(\mathcal H=\ker\eta\) and \(I_1=\varphi|_{\mathcal H}\).  The
characteristic connection preserves \(g\), \(\xi\) and \(I_1\), so its
holonomy on the oriented rank-four bundle \(\mathcal H\) lies in \(U(2)\).
The vanishing of \(\rho^\nabla\) is precisely the vanishing of the curvature
of the determinant connection.  Thus the restricted holonomy lies in
\(SU(2)\).  Since \(M\) is simply connected, full and restricted holonomy
coincide, and consequently \(\mathcal H\) admits a global
\(\nabla\)-parallel quaternionic triple
\begin{equation}\label{eq:ST-quaternionic-triple}
 I_1,I_2,I_3,
 \qquad
 I_aI_b=-\delta_{ab}+\varepsilon_{abc}I_c.
\end{equation}
For \(q=(q_1,q_2,q_3)\in S^2\), set
\[
 I_q=q_1I_1+q_2I_2+q_3I_3
\]
and define
\begin{equation}\label{eq:ST-Jq-definition}
 J_q|_{\mathcal H}=I_q,
 \qquad
 J_q\xi=\partial_s,
 \qquad
 J_q\partial_s=-\xi.
\end{equation}
The product connection
\[
 D=d_{S^1}\oplus\nabla
\]
preserves both \(G\) and every \(J_q\), and its torsion is the pull-back of
\(H\).  Hence \(D\) is a Hermitian connection with totally skew torsion.  By
the uniqueness theorem for the characteristic connection
\cite{FriedrichIvanov2002}, \(D\) is the characteristic connection of
\((G,J_q)\).  In particular \(\widetilde N_{J_q}\) is a three-form and
\(T^B_{J_q}=H\).  The strong Sasaki-with-torsion hypothesis gives \(dH=0\),
proving the first three assertions in \eqref{eq:ST-rotation-properties}.
Moreover $\operatorname{Hol}(D)\subset SU(2)\subset SU(3),$
where the last inclusion is taken with respect to any \(J_q\).  Therefore the
first Bismut Ricci form vanishes.  Since \(D\) is the product connection, its
curvature has no circle component and \eqref{eq:ST-rotation-curvature}
follows.

It remains to determine the integrable parameters.  Let \(F_a\) be the
fundamental two-form of \(I_a\) on \(\mathcal H\), with the orientation for
which \(F_1,F_2,F_3\) span \(\Lambda^+\mathcal H^*\).  Normality of the
Sasaki-with-torsion structure implies that \(d\eta\) is of type \((1,1)\)
with respect to \(I_1\).  Thus pointwise
\begin{equation}\label{eq:ST-deta-decomp}
 d\eta=aF_1+\gamma_-,
 \qquad
 \gamma_-\in\Lambda^-\mathcal H^*,
\end{equation}
where \(a\neq0\) by \eqref{eq:ST-rotation-hyp}.  For the orthogonal complex
structure \(I_q\), the anti-self-dual summand is contained in
\(\Lambda^{1,1}_{I_q}\), while
\[
 F_q=q_1F_1+q_2F_2+q_3F_3
\]
spans the self-dual \((1,1)\)-line.  Consequently the projection of \(F_1\)
to the real \((2,0)+(0,2)\) subspace for \(I_q\) has squared norm
proportional to \(1-q_1^2\).  Hence
\begin{equation}\label{eq:ST-deta-02}
 (d\eta)^{0,2}_{I_q}\neq0
 \quad\Longleftrightarrow\quad
 q\neq\pm(1,0,0).
\end{equation}
On \(X\), the one-form \(\zeta=\eta+i\,ds\) is of type \((1,0)\) for
\(J_q\) and \(d\zeta=d\eta\).  Thus \eqref{eq:ST-deta-02} violates the
Newlander--Nirenberg type condition \(d\Omega^{1,0}\subset
\Omega^{2,0}\oplus\Omega^{1,1}\), proving that \(J_q\) is non-integrable for
all \(q\neq\pm(1,0,0)\).  At the two poles \(I_q=\pm I_1\), normality gives
the usual integrable product complex structures.  This proves
\eqref{eq:ST-rotation-nonintegrable}.
\end{proof}

 \begin{remark}
The Riemannian metric is not modified by the construction: the quaternionic
rotation produces an $S^2$-family of compatible almost complex structures on
a fixed Riemannian manifold, all sharing the same characteristic connection,
torsion and characteristic curvature.  In particular these examples are
non-integrable deformations of the almost complex structure at fixed metric,
which is precisely the situation studied in Section~\ref{Section6}.
\end{remark}
 The $L^{a,b,c}$ manifolds considered as strong $\nabla$-Einstein
five-manifolds in \cite{BrienzaFinoFowdarGrantcharov2026} are regular when
their parameters satisfy the regularity and coprimality conditions of
\cite[Section~4.2]{CouzensGauntlettMartelliSparks2019}. For these parameter
values, $L^{a,b,c}$ is a smooth compact manifold diffeomorphic to
$S^2\times S^3$. Applying the previous construction therefore gives
compact non-Bismut-flat examples on $S^1\times S^2\times S^3$.
\begin{cor}
\label{cor:Labc-quaternionic-examples}
For every such regular manifold $L^{a,b,c}$, the compact six-manifold
\begin{equation}\label{eq:Labc-topology-six}
 S^1\times L^{a,b,c}\cong S^1\times S^2\times S^3
\end{equation}
carries an \(S^2\setminus\{\pm q_0\}\)-family of almost Hermitian structures
satisfying
\begin{equation}\label{eq:Labc-main-properties}
 \widetilde N\in\Omega^3,
 \qquad
 \widetilde N\neq0,
 \qquad
 dT^B=0,
 \qquad
 \rho^B=0,
 \qquad
 R^B\neq0.
\end{equation}
Thus these are  almost-SKT and almost-CYT structures with  nowhere vanishing skew Nijenhuis tensor and  non-flat
characteristic connection.
\end{cor}

\begin{proof}
{ The regular manifolds $L^{a,b,c}$ considered in
\cite{BrienzaFinoFowdarGrantcharov2026} form a family of strong
$\nabla$-Einstein five-manifolds for which the constant in the strong
$\nabla$-Einstein condition is equal to $2$; see
Theorem~5.12 and the discussion following Remark~5.13 in
\cite{BrienzaFinoFowdarGrantcharov2026}.}
These are the GK
\(L^{a,b,c}\) geometries of \cite[Section~4.2]{CouzensGauntlettMartelliSparks2019},
not the Sasaki--Einstein family bearing the same notation.  The regularity
analysis in that reference gives
\begin{equation}\label{eq:Labc-integers}
 -a+b=c+d,
 \qquad
 b>c\ge d>0,
 \qquad
 b>a>0,
\end{equation}
with the stated coprimality assumptions, and shows that
\(L^{a,b,c}\cong S^2\times S^3\).  In particular \(L^{a,b,c}\) is simply
connected, so Theorem~\ref{thm:ST-quaternionic-rotation} applies.

It remains only to check non-flatness.  In the notation of
\cite{CouzensGauntlettMartelliSparks2019}, for complex dimension three the
transverse scalar curvature $R$ is related to the conformal factor by
$e^h=R/8$, while on the regular $L^{a,b,c}$ family
\[
 e^h=\frac{1+a_1}{4(u-v)}.
\]
Hence
\begin{equation}\label{eq:Labc-transverse-scalar}
  R=\frac{2(1+a_1)}{u-v}
\end{equation}
is non-constant,  where $h,a_1,u,v$ denote the quantities written
$B,\alpha,\eta,\xi$ in \cite[Section~4.2]{CouzensGauntlettMartelliSparks2019}. The non-constant case of
\cite[Theorem~5.16]{BrienzaFinoFowdarGrantcharov2026} is non-flat.  Thus
\(R^\nabla\neq0\), and \eqref{eq:ST-rotation-curvature} gives \(R^B\neq0\).
\end{proof}

\begin{remark}
Corollary~\ref{cor:Labc-quaternionic-examples} answers affirmatively the question,
raised in the discussion following Example~4.12 in
\cite{FinoFowdar2026}, of whether there exist non-Bismut-flat
almost-SKT and almost-CYT six-manifolds with $N_J\neq0$. Notice that
each regular $L^{a,b,c}$ produces a full two-parameter sphere of almost
Hermitian structures with the same characteristic connection, closed
torsion, and curvature; only the two antipodal members are integrable.
\end{remark}

\begin{remark}
For instance the integers
$(a,b,c,d)=(1,5,3,1)$
satisfy \eqref{eq:Labc-integers} and the coprimality assumptions of
\cite{CouzensGauntlettMartelliSparks2019}.  Thus
\(S^1\times L^{1,5,3}\cong S^1\times S^2\times S^3\) gives a concrete compact
family satisfying \eqref{eq:Labc-main-properties}.
\end{remark}
A second quaternionic-rotation construction applies directly to compact
simply connected quasi-regular BHE threefolds.
\begin{thm}
\label{thm:BHE-threefold-quaternionic-rotation}
Let \((N^6,g,J_1)\) be a compact simply connected non-K\"ahler quasi-regular
Bismut--Hermitian--Einstein threefold.  Let
\[
 \mathcal V=\operatorname{span}\{V,W\},\qquad W=J_1V,
 \qquad TN=\mathcal V\oplus\mathcal H,
\]
be its canonical Bismut-parallel splitting.  Then \(\mathcal H\) admits a
global \(\nabla^B\)-parallel quaternionic triple \((I_1,I_2,I_3)\), with
\(I_1=J_1|_{\mathcal H}\).  For every \(q=(q_1,q_2,q_3)\in S^2\), define
\begin{equation}\label{eq:BHE-threefold-Jq}
 I_q=q_1I_1+q_2I_2+q_3I_3,
 \qquad
 J_q|_{\mathcal H}=I_q,
 \qquad
 J_qV=W,\quad J_qW=-V.
\end{equation}
Then the original Bismut connection is the characteristic connection of
\((g,J_q)\) for every \(q\), and
\begin{equation}\label{eq:BHE-threefold-rotation-properties}
 \widetilde N_{J_q}\in\Omega^3(N),
 \qquad
 T^B_{J_q}=T^B_{J_1},
 \qquad
 dT^B_{J_q}=0,
 \qquad
 \rho^B_{J_q}=0,
 \qquad
 R^B_{J_q}=R^B_{J_1}.
\end{equation}
Moreover
\begin{equation}\label{eq:BHE-threefold-nonintegrable}
 N_{J_q}(p)\neq0
\quad
\text{for every }p\in N
\text{ and every }q\in S^2\setminus\{\pm(1,0,0)\}.
\end{equation}
Thus, if the initial BHE connection is non-flat, every non-integrable rotated
member is a strong ABHE structure which is not Bismut-flat.
\end{thm}

\begin{proof}
For a compact non-K\"ahler BHE threefold the canonical vector field \(V\) is
\(\nabla^B\)-parallel; see the structural discussion preceding
\cite[Theorem~1.1]{ApostolovLahdiliLee2026}.  Since \(\nabla^BJ_1=0\), the field
\(W=J_1V\) is parallel as well.  Hence \(\mathcal V\) and \(\mathcal H\) are
\(\nabla^B\)-parallel, and the holonomy on \(\mathcal H\) lies in \(U(2)\).
The BHE condition \(\rho^B=0\) says that the determinant curvature vanishes,
so the restricted holonomy on \(\mathcal H\) lies in \(SU(2)\).  Since \(N\)
is simply connected, full and restricted holonomy coincide.  The real
four-dimensional standard representation of \(SU(2)\) is quaternionic;
therefore \(\mathcal H\) carries a global parallel triple
\(I_1,I_2,I_3\) satisfying the quaternionic relations.

The same connection \(\nabla^B\) preserves \(g\), the vertical complex line,
and every \(I_q\), hence it preserves every \(J_q\).  Its torsion is the
original three-form \(T^B_{J_1}\).  Thus it is a metric \(J_q\)-connection
with totally skew torsion.  By uniqueness of the characteristic connection
\cite{FriedrichIvanov2002}, it is the characteristic connection of
\((g,J_q)\), proving the first two statements in
\eqref{eq:BHE-threefold-rotation-properties}.  The original BHE metric is
pluriclosed, hence \(dT^B_{J_1}=0\), so the strong condition is unchanged by
the rotation.  Moreover the holonomy acts trivially on the vertical complex
line and through \(SU(2)\) on \((\mathcal H,I_q)\).  Consequently it is
contained in \(SU(3)\) for every \(J_q\), giving \(\rho^B_{J_q}=0\).  The
connection itself has not changed, so neither has its curvature.

It remains to prove non-integrability.  In the quasi-regular description of
\cite{ApostolovLahdiliLee2026}, let \(\eta_V,\eta_W\) be the two connection
one-forms of the vertical torus.  On the horizontal K\"ahler orbifold one has
\begin{equation}\label{eq:BHE-threefold-curvatures}
 d\eta_V=\theta^T,
 \qquad
 d\eta_W=-\rho(\omega_K^T),
 \qquad
 g_H=\frac{R}{2}g_K^T,
\end{equation}
where \(R>0\), \(\theta^T\) is primitive of type \((1,1)\), and \(\rho\) is
the transverse Ricci form.  Let \(F_1\) be the fundamental form of
\((g_H,I_1)\).  Since \(F_1=(R/2)\omega_K^T\), the K\"ahler-surface Ricci
decomposition gives
\begin{equation}\label{eq:BHE-threefold-Ricci-decomp}
 \rho(\omega_K^T)=\frac{R}{4}\omega_K^T+\rho_0
 =\frac12F_1+\rho_0,
 \qquad
 \rho_0\in\Lambda^-\mathcal H^*.
\end{equation}
With the quaternionic orientation, both \(\theta^T\) and \(\rho_0\) are
anti-self-dual and hence remain primitive \((1,1)\)-forms for every \(I_q\).
If
\[
 F_q=q_1F_1+q_2F_2+q_3F_3
\]
is the fundamental form of \(I_q\), then the self-dual
\((2,0)+(0,2)\)-part of \(F_1\) is \(F_1-q_1F_q\). Therefore 
\begin{equation}\label{eq:BHE-threefold-curvature-02}
 \bigl(d\eta_W\bigr)^{2,0+0,2}_{I_q}
 =-\frac12\bigl(F_1-q_1F_q\bigr),
\end{equation}
which is nonzero exactly when \(q\neq\pm(1,0,0)\).

This conclusion does not require the rotated \(I_q\) to descend to the
orbifold quotient.  Indeed the vertical complex structure is fixed, so (up
to the harmless choice of sign convention) the global one-form
\[
 \zeta=\eta_V+i\eta_W
\]
is of type \((1,0)\) for every \(J_q\).  Since
\(d\eta_V=\theta^T\) stays of type \((1,1)\), whereas
\eqref{eq:BHE-threefold-curvature-02} gives a nonzero
\((0,2)\)-component of \(d\eta_W\), the form \(d\zeta\) has a nonzero
\((0,2)\)-part whenever \(q\neq\pm(1,0,0)\).  The Newlander--Nirenberg
criterion then shows that \(J_q\) is non-integrable.  Since its lowered
Nijenhuis tensor is already known to be a three-form, this proves
\eqref{eq:BHE-threefold-nonintegrable}.  At the two poles one recovers the
original horizontal complex structure and its negative, for which the two
curvature forms are \((1,1)\).  This completes the proof.
\end{proof}

\begin{cor}[Non-Bismut-flat almost-CYT families]
\label{cor:BHE-nonflat-families}

Consider any of the following compact simply connected BHE threefolds: \begin{enumerate} \item $S^3\times S^3$, endowed with any of the quasi-regular non-Samelson BHE metrics constructed in \cite[Theorem~1.6]{ApostolovLahdiliLee2026}; \item either of the manifolds 
\begin{equation*}
    \begin{split}
         N_1=& \bigl(\#^{3}(S^2\times S^4)\bigr) \mathbin{\#} \bigl(\#^{4}(S^3\times S^3)\bigr), \\
         N_2=& \bigl(\#^{5}(S^2\times S^4)\bigr) \mathbin{\#} \bigl(\#^{6}(S^3\times S^3)\bigr),
    \end{split}
\end{equation*}

endowed with the regular BHE structures constructed in \cite[Theorem~3(a),(b)]{ALLR2026}. \end{enumerate} Then, with the metric fixed, each of these BHE structures gives rise to an $S^2$-family of $g$-orthogonal almost complex structures $\{J_q\}_{q\in S^2}$ having the same characteristic connection. Exactly two antipodal members are integrable. For every $ q\in S^2\setminus\{\pm q_0\}, $ the corresponding almost Hermitian structure satisfies \[ \widetilde N_{J_q}\in\Omega^3, \qquad \widetilde N_{J_q}(p)\neq0 \quad\text{for every }p, \] and \[ dT^B_{J_q}=0, \qquad \rho^B_{J_q}=0, \qquad R^B_{J_q}\not\equiv0. \] In particular, all these manifolds admit non-integrable almost-SKT and almost-CYT structures with nowhere vanishing skew Nijenhuis tensor and non-flat characteristic connection.
\end{cor}

\begin{proof}
The manifold $S^3\times S^3$ is simply connected, and the metrics of
\cite[Theorem~1.6]{ApostolovLahdiliLee2026} are compact quasi-regular
non-K\"ahler BHE metrics with non-constant transverse scalar curvature
which are not locally isometric to a Samelson geometry. Hence
Theorem~\ref{thm:BHE-threefold-quaternionic-rotation} applies.

If the corresponding Bismut connection were flat, the classification
theorem of Wang--Yang--Zheng \cite{WangYangZheng2020} would imply that
a simply connected compact non-K\"ahler Bismut-flat threefold is a
central Calabi--Eckmann threefold, hence a Samelson geometry. This
contradicts the preceding property. Therefore
$ R^B\not\equiv0$. Since quaternionic rotation leaves the characteristic connection
unchanged, the same holds for every rotated member.

For the second family, by
\cite[Theorem~3(a),(b)]{ALLR2026}, each $N_i$ admits a non-K\"ahler
regular BHE structure, hence in particular a
quasi-regular one. Since both manifolds
are simply connected,
Theorem~\ref{thm:BHE-threefold-quaternionic-rotation} applies and
gives the asserted quaternionic family, including the characterization
of its integrable members.

It remains only to verify non-flatness. By the classification of
compact Bismut-flat manifolds in \cite{WangYangZheng2020}, a compact
simply connected Bismut-flat complex threefold is diffeomorphic to
$S^3\times S^3$. However,
\[
 b_2(N_1)=3,
 \qquad
 b_2(N_2)=5,
 \qquad
 b_2(S^3\times S^3)=0.
\]
Thus the initial Bismut connection on each $N_i$ is not flat.
Again, since quaternionic rotation leaves the characteristic
connection unchanged, every member of the corresponding family has
non-flat characteristic connection.
\end{proof}

\begin{remark}
Corollaries~\ref{cor:Labc-quaternionic-examples} and
\ref{cor:BHE-nonflat-families} have the same structural feature: the almost
complex structure moves in a quaternionic two-sphere while the metric,
characteristic connection, torsion and characteristic curvature remain
fixed.  The non-integrability is detected by a curvature form which is
\((1,1)\) at the two poles and acquires a nonzero \((2,0)+(0,2)\)-part away
from them.
\end{remark}

\section{Lifting to generalized Ricci flow with moving $J$} \label{Section6}
For the particular projected flow considered in this paper there is also a
weaker, flow-dependent stationary condition
\begin{equation}\label{eq:projected-ABHE}
\mathbb P_J\bigl((\rho^B_\omega)^{1,1}\bigr)=0.
\end{equation}
Every strong ABHE metric satisfies \eqref{eq:projected-ABHE}, but the converse
need not hold because the projection can annihilate a nonzero curvature
component.

The almost-SKT constraint has a direct interpretation in generalized Ricci
flow.  Associate to an admissible form the generalized metric datum
\begin{equation}\label{eq:iota-generalized-metric}
\iota_J(\omega):=(g_\omega,H_\omega),
\qquad H_\omega:=T^B_\omega.
\end{equation}
Then \(dH_\omega=0\), and for a variation \(\alpha\) one has
\begin{equation}\label{eq:d-iota-generalized-metric}
D\iota_J(\alpha)
=
\left(\mathcal I^{-1}(\alpha),
      d^c\alpha+\mathcal N_J(\alpha)\right).
\end{equation}
Thus the admissible space is naturally mapped into the configuration space of
generalized Ricci flow, whose points are pairs \((g,H)\) with \(H\) closed. Recall that the generalized Ricci flow is given by
\eqref{eq:generalized-Ricci-flow}.
\subsection{The metric identity and the infinitesimal generalized-Ricci lift}

The relation with generalized Ricci flow can be made considerably more precise.
It is useful first to remove a possible source of confusion: the Bismut--Ricci
form appearing in the constrained flow already contains the correct generalized
Ricci metric tensor, up to the Lee-form diffeomorphism gauge.  No integrability
of $J$ is needed for this statement.

Let
\begin{equation}\label{eq:alphaB-hB-dynamic}
\alpha_B:=-(\rho^B)^{1,1},
\qquad
h_B:=\mathcal I^{-1}(\alpha_B).
\end{equation}
{With this notation, the metric identity takes the following form.}
\begin{prop}
\label{prop:almost-G1-metric-GRF-identity}
Let $(M^{2n},g,J)$ be of class $G_1$, so that the characteristic torsion $T$ is
a $3$-form, and assume $dT=0$.  Then
\begin{equation}\label{eq:almost-G1-metric-GRF-identity}
 h_B
 =-\operatorname{Ric}^g+\frac14T^2
   -\frac12\mathcal L_{\theta^\sharp}g.
\end{equation}
 Thus, up to the Lee-form diffeomorphism gauge and a constant rescaling of
time, the unprojected Bismut--Ricci evolution agrees with the metric equation
of  the generalized Ricci flow.
\end{prop}

\begin{proof}
Put
\[
S:=\operatorname{Ric}^B+\nabla^B\theta.
\]
Since $dT=0$, identity~\eqref{eq:rhoB-RicB} gives
$\rho^B(X,Y)=S(X,JY)$.  As $\rho^B$ is a two-form,
\[
(\rho^B)^{1,1}(X,JY)
=-\frac12\bigl(S(X,Y)+S(Y,X)\bigr).
\]
With the convention \eqref{eq:I-final},
$h_B(X,Y)=(\rho^B)^{1,1}(X,JY)$, hence
\[
h_B=-\operatorname{Sym}\operatorname{Ric}^B
    -\operatorname{Sym}(\nabla^B\theta).
\]
The skew part of the characteristic Ricci tensor is
$-(1/2)\delta T$, while \eqref{eq:Ricci-comparison}
yields
\[
\operatorname{Sym}\operatorname{Ric}^B
=\operatorname{Ric}^g-\frac14T^2.
\]
Finally the symmetric parts of $\nabla^B\theta$ and $\nabla^g\theta$ agree. Hence
\[
\operatorname{Sym}(\nabla^B\theta)=\frac12\mathcal L_{\theta^\sharp}g.
\]
This proves \eqref{eq:almost-G1-metric-GRF-identity}.
\end{proof}
 For the comparison with the Bismut--Ricci evolution, we use the equivalent
time parametrization of \eqref{eq:generalized-Ricci-flow} in which the
right-hand sides are divided by \(2\). This is the parametrization naturally
paired with pluriclosed flow in
\cite[Proposition~9.8 and Remark~9.11]{GarciaFernandezStreets2020}. Thus
\begin{equation}\label{eq:half-speed-GRF}
\partial_tg=\mathcal G(g,H):=-\operatorname{Ric}^g+\frac14H^2,
\qquad
\partial_tH=\mathcal Q(g,H):=-\frac12dd_g^*H.
\end{equation}
 Proposition~\ref{prop:almost-G1-metric-GRF-identity} says
\begin{equation}\label{eq:metric-GRF-Lee-decomposition}
\mathcal G(g,T)=h_B+\frac12\mathcal L_{\theta^\sharp}g.
\end{equation}

The remaining issue is therefore entirely torsional.  Define the fixed-$J$
torsion velocity associated with the unprojected Bismut--Ricci direction by
\begin{equation}\label{eq:tauB-fixed-J}
\tau_B
:=d^c\alpha_B+\mathcal N_J(\alpha_B).
\end{equation}
After applying the Lee diffeomorphism gauge generated by
$\frac12\theta^\sharp$, the induced torsion velocity is
\[
\tau_B+\frac12\mathcal L_{\theta^\sharp}T.
\]
Accordingly, introduce the \emph{torsion-lifting defect}
\begin{equation}\label{eq:torsion-lifting-defect}
\mathfrak D_J
:=2\tau_B+\mathcal L_{\theta^\sharp}T+dd_g^*T.
\end{equation}
Then the unprojected fixed-$J$ Bismut--Ricci evolution, followed by the Lee
form diffeomorphism gauge, induces generalized Ricci flow on $(g,T)$ if and
only if
\begin{equation}\label{eq:defect-zero-fixed-J}
\mathfrak D_J=0.
\end{equation}
In the integrable pluriclosed case this identity is exactly the content of the
gauge equivalence in \cite[Proposition~9.8 and Corollary~9.9]
{GarciaFernandezStreets2020}.

For a genuinely non-integrable structure one can allow the almost complex
structure itself to absorb the defect.  To formulate this invariantly, define
for every compatible almost Hermitian pair the covariant three-tensor
\begin{equation}\label{eq:T-hat-extension}
\widehat T(g,J):=d_J^c\omega+\widetilde N_{g,J}.
\end{equation}
On the $G_1$ locus this is precisely the characteristic torsion three-form.
If $K\in\operatorname{End}(TM)$ satisfies
\begin{equation}\label{eq:K-tangent-OJ}
KJ+JK=0,
\qquad
K^{*,g}=-K,
\end{equation}
then $K$ is tangent to the space of $g$-orthogonal almost complex structures.
We define the first-order \emph{torsion response operator}
\begin{equation}\label{eq:RJK-operator}
\mathscr R_{g,J}K
:=D_J\widehat T(g,J)[K].
\end{equation}
The operator is first order because both $d_J^c\omega$ and $N_J$ contain one
derivative of $J$.

\subsection{Ellipticity of the moving-$J$ skew-torsion constraint}

The analytic nature of the $J$-lifting equation can be described more
precisely.  Let
\[
\mathcal K_J
:=\{K\in\operatorname{End}(TM):KJ+JK=0,\ K^{*,g}=-K\}
\]
be the bundle tangent to the space of $g$-orthogonal almost complex
structures.  Since the tensor
\(\widetilde N_{g,J}(X,Y,Z)=g(N_J(X,Y),Z)\) is always skew in its first two
arguments, define the quotient bundle
\[
\mathcal C_J
:=\bigl(\Lambda^2T^*M\otimes T^*M\bigr)/\Lambda^3T^*M
\]
and let $q$ denote the quotient map.  The infinitesimal failure of the
Nijenhuis tensor to remain totally skew under a fixed-metric variation of
$J$ is the first-order operator
\begin{equation}\label{eq:skew-defect-J-operator}
\mathscr S_{g,J}K
:=q\!\left(D_J\widetilde N_{g,J}[K]\right)
:\Gamma(\mathcal K_J)\longrightarrow\Gamma(\mathcal C_J).
\end{equation}
Because $D_J(d_J^c\omega)[K]$ is a genuine three-form, one also has
\begin{equation}\label{eq:S-qR}
\mathscr S_{g,J}=q\circ\mathscr R_{g,J}.
\end{equation}
{The following result shows that preserving the skew-Nijenhuis condition
under a fixed-metric deformation of $J$ is an overdetermined elliptic
problem.}
\begin{prop}
\label{prop:J-G1-overdetermined-elliptic}
At every almost Hermitian structure, the first-order operator
$\mathscr S_{g,J}$ has injective principal symbol for every nonzero real
covector.  Hence it is overdetermined elliptic.  In particular,
$\mathscr S_{g,J}^{*,g}\mathscr S_{g,J}$ is elliptic of order two.
\end{prop}

\begin{proof}
Complexify at a point and choose a unitary frame
$\{Z_1,\ldots,Z_n\}$ of $T^{1,0}M$.  An element
$K\in\mathcal K_J$ is determined by
\[
K Z_{\bar i}=\mu_{\bar i}^{\ k}Z_k,
\]
and the orthogonality condition implies that, after lowering the holomorphic
index,
\[
\mu_{\bar i\bar j}:=g_{k\bar j}\mu_{\bar i}^{\ k}
\]
is skew-symmetric in $\bar i,\bar j$.  Thus the complexified symbol variable
is naturally a $(0,2)$-form $\mu$.

Let $0\neq\xi\in T_x^*M$ be real and put
$\zeta:=\xi^{0,1}\neq0$.  Up to a universal nonzero constant depending only
on the normalization of $N_J$, the $(0,2)\otimes(0,1)$ component of the
principal symbol of the linearized lowered Nijenhuis tensor is
\begin{equation}\label{eq:symbol-linearized-N-J}
\bigl(\sigma_1(D_J\widetilde N)(\xi)\mu\bigr)_{\bar i\bar j\bar k}
=
\zeta_{\bar i}\mu_{\bar j\bar k}
-
\zeta_{\bar j}\mu_{\bar i\bar k}.
\end{equation}
The symbol of $\mathscr S_{g,J}$ is the class of
\eqref{eq:symbol-linearized-N-J} modulo $\Lambda^{0,3}$.

Choose the unitary frame so that $\zeta$ is a nonzero multiple of
$\bar\varepsilon^1$.  Suppose that this class vanishes.  Then
\eqref{eq:symbol-linearized-N-J} must be totally alternating.  For
$a,b>1$ one has
\[
\bigl(\sigma_1(D_J\widetilde N)(\xi)\mu\bigr)_{\bar1\bar a\bar b}
=c\,\mu_{\bar a\bar b},
\qquad
\bigl(\sigma_1(D_J\widetilde N)(\xi)\mu\bigr)_{\bar a\bar b\bar1}=0,
\]
with $c\neq0$.  Total skew-symmetry therefore gives
$\mu_{\bar a\bar b}=0$.  Taking a repeated $\bar1$ index gives
\[
\bigl(\sigma_1(D_J\widetilde N)(\xi)\mu\bigr)_{\bar1\bar a\bar1}
=c\,\mu_{\bar a\bar1}=0,
\]
so also $\mu_{\bar1\bar a}=0$.  Hence $\mu=0$, proving injectivity of the
symbol.  The last assertion follows because the principal symbol of
$\mathscr S^{*}\mathscr S$ is
$\sigma_1(\mathscr S)^*\sigma_1(\mathscr S)$ and is positive definite for
$\xi\neq0$.
\end{proof}
{ On a compact manifold, the ellipticity of $\mathcal S_{g,J}$ has the
following finite-dimensional consequence.}
\begin{cor}
\label{cor:finite-dimensional-J-G1-deformations}
If $M$ is compact, then
\begin{equation}\label{eq:Z1-G1-J}
\mathcal Z^1_{G_1}(g,J)
:=\ker\mathscr S_{g,J}
\end{equation}
is finite dimensional and consists of smooth sections.  It is the space of
infinitesimal orthogonal variations of $J$ for which the lowered Nijenhuis
tensor remains a three-form to first order.
\end{cor}

\begin{proof}
The kernel agrees with the kernel of the elliptic operator
$\mathscr S_{g,J}^{*,g}\mathscr S_{g,J}$.  Standard elliptic theory on a
compact manifold gives the conclusion.
\end{proof}
Relation~\eqref{eq:S-qR} also gives the corresponding
ellipticity statement for the full torsion response operator.
\begin{cor}
\label{cor:R-J-overdetermined-elliptic}
The torsion response operator $\mathscr R_{g,J}$ itself has injective
principal symbol.  Thus $\mathscr R_{g,J}^{*,g}\mathscr R_{g,J}$ is elliptic
of order two. Consequently, every distributional
solution of $\mathscr R_{g,J}K=F$ with smooth right-hand side
is smooth. If $M$ is compact, then $\ker\mathscr R_{g,J}$
is finite dimensional, and the solutions are unique modulo
 this kernel.
\end{cor}

\begin{proof}
If $\sigma_1(\mathscr R)(\xi)K=0$, then
\eqref{eq:S-qR} gives
$\sigma_1(\mathscr S)(\xi)K=0$, hence $K=0$ by
Proposition~\ref{prop:J-G1-overdetermined-elliptic}. Therefore
$\mathscr R_{g,J}^{*,g}\mathscr R_{g,J}$ is elliptic.

If $\mathscr R_{g,J}K=F$ distributionally and $F$ is smooth,
then
\[
\mathscr R_{g,J}^{*,g}\mathscr R_{g,J}K
=
\mathscr R_{g,J}^{*,g}F,
\]
so elliptic regularity implies that $K$ is smooth.
If $M$ is compact, integration by parts gives
\[
\ker\bigl(\mathscr R_{g,J}^{*,g}\mathscr R_{g,J}\bigr)
=
\ker\mathscr R_{g,J},
\]
and this kernel is finite dimensional by elliptic theory.
Finally, any two solutions with the same right-hand side
differ by an element of $\ker\mathscr R_{g,J}$.
\end{proof}

Overdetermined ellipticity gives regularity and a finite-dimensional kernel,
but not surjectivity; solvability of the inhomogeneous lifting equation still
requires compatibility conditions.

For comparison, in the unprojected equation the right-hand side need not be a
three-form because $\alpha_B$ need not preserve the skew-Nijenhuis
constraint.  Applying $q$ to
\eqref{eq:unprojected-lifting-equation} gives the first necessary elliptic
compatibility equation
\begin{equation}\label{eq:unprojected-skew-lift-equation}
\mathscr S_{g,J}K=-q(\tau_B).
\end{equation}
Thus even before matching the three-form component of the generalized Ricci
torsion equation, the non-integrable lift must solve an overdetermined
elliptic equation.

\begin{prop}[Unprojected flow]
\label{prop:exact-moving-J-unprojected-lift}
Assume $dT=0$ and let $\alpha_B,h_B,\tau_B$ be as above.  Consider
\begin{equation}\label{eq:unprojected-moving-J-system}
\partial_tg
=h_B+\frac12\mathcal L_{\theta^\sharp}g,
\qquad
\partial_tJ
=\frac12\mathcal L_{\theta^\sharp}J+K,
\end{equation}
where $K$ satisfies \eqref{eq:K-tangent-OJ}.  Then the induced pair
$(g,T=\widehat T(g,J))$ satisfies 
\eqref{eq:half-speed-GRF} if and only if
\begin{equation}\label{eq:unprojected-lifting-equation}
\mathscr R_{g,J}K=-\frac12\mathfrak D_J.
\end{equation}
In particular, $K=0$ is sufficient precisely when the torsion-lifting defect
vanishes.
\end{prop}

\begin{proof}
The metric equation is exactly \eqref{eq:metric-GRF-Lee-decomposition}.  By
naturality of $\widehat T$ under diffeomorphisms, differentiation of
\eqref{eq:unprojected-moving-J-system} gives
\[
\partial_tT
=\tau_B+\frac12\mathcal L_{\theta^\sharp}T
 +\mathscr R_{g,J}K.
\]
Equating this with the second equation of \eqref{eq:half-speed-GRF} is
precisely \eqref{eq:unprojected-lifting-equation}.  Since the right-hand side
of the generalized Ricci $H$-equation is a three-form, the same equation also
forces the infinitesimal skew-torsion defect to vanish.  In particular, along
a smooth lift satisfying the equation at every time, the skew defect remains
zero if it vanishes initially.
\end{proof}

The previous proposition concerns the unprojected Bismut--Ricci direction.  It
is also useful to record the exact obstruction created by the projection used
in \eqref{eq:Bismut-Ricci-constrained-final}.  Put
\begin{equation}\label{eq:projected-alpha-h-tau}
\alpha_c:=\mathbb P_J\alpha_B,
\qquad
h_c:=\mathcal I^{-1}(\alpha_c),
\qquad
\tau_c:=d^c\alpha_c+\mathcal N_J(\alpha_c),
\end{equation}
and define the metric projection defect
\begin{equation}\label{eq:metric-projection-defect}
\mathfrak P_J
:=h_B-h_c
=\mathcal I^{-1}\!\left((\operatorname{Id}-\mathbb P_J)\alpha_B\right).
\end{equation}
Then
\begin{equation}\label{eq:projected-metric-GRF-decomposition}
\mathcal G(g,T)
=h_c+\frac12\mathcal L_{\theta^\sharp}g+\mathfrak P_J.
\end{equation}
Thus a variation of $J$ alone can never remove $\mathfrak P_J$, since it does
not change the metric.  The most general diffeomorphism correction is obtained
by setting
\begin{equation}\label{eq:X-Lee-plus-Y}
X=\frac12\theta^\sharp+Y
\end{equation}
and asking that
\begin{equation}\label{eq:metric-gauge-lifting-equation}
\mathcal L_Yg=\mathfrak P_J.
\end{equation}
This is the first lifting equation.

\begin{thm}[Infinitesimal constrained lift criterion]
\label{thm:exact-GRF-lift-constrained}
Let $(g,J)$ satisfy the almost-SKT constraints, and let
$h_c$ and $\tau_c$ be the metric and torsion variations
determined by the constrained Bismut--Ricci equation
\eqref{eq:Bismut-Ricci-constrained-final}.
At a fixed time, there exists an infinitesimal variation
with varying $J$ of the form
\begin{equation}\label{eq:general-constrained-GRF-lift}
\partial_tg=h_c+\mathcal L_Xg,
\qquad
\partial_tJ=\mathcal L_XJ+K,
\end{equation}
with $K$ satisfying \eqref{eq:K-tangent-OJ}, whose induced
metric and torsion variations satisfy
\[
\partial_tg=-\operatorname{Ric}^g+\frac14T^2,
\qquad
\partial_tT=-\frac12dd_g^*T,
\]
if and only if one can solve simultaneously \begin{align}
\mathcal L_{X-\frac12\theta^\sharp}g
&=\mathfrak P_J,
\label{eq:lift-system-metric}\\
\mathscr R_{g,J}K
&=-\frac12dd_g^*T-\tau_c-\mathcal L_XT.
\label{eq:lift-system-torsion}
\end{align}
Equivalently, after writing $X=\frac12\theta^\sharp+Y$,
the first equation is
\eqref{eq:metric-gauge-lifting-equation}
and the second is a first-order lifting equation for
the orthogonal almost complex variation $K$.
\end{thm}

\begin{proof} 
All variations are evaluated at the given almost-SKT pair.
Equation \eqref{eq:projected-metric-GRF-decomposition}
shows that the metric variation in
\eqref{eq:general-constrained-GRF-lift} agrees with the
generalized Ricci flow metric equation, in the chosen
time parametrization, if and only if
\eqref{eq:lift-system-metric} holds.
By the naturality of $\widehat T$, the induced torsion
variation is
\[
\partial_tT
=
\tau_c+\mathcal L_XT+\mathscr R_{g,J}K.
\]
Thus $\partial_tT=-\frac12dd_g^*T$ is equivalent to
\eqref{eq:lift-system-torsion}.
The compatibility conditions \eqref{eq:K-tangent-OJ}
ensure that the additional $K$-variation is tangent
to the space of $g$-orthogonal almost complex structures.
\end{proof}
 Note that the criterion is infinitesimal and does not by itself
imply the existence of a smooth time-dependent lift.

For the projected constrained flow, the previous criterion has the
following finite-dimensional consequence.
\begin{prop}
\label{prop:projected-J-lift-finite-dimensional}
Assume $M$ is compact and let $X$ solve the metric lifting equation
\eqref{eq:lift-system-metric}.  For the  constrained Bismut–Ricci variation,

set
\[
F_X:=-\frac12dd_g^*T-\tau_c-\mathcal L_XT.
\]
Then $F_X$ is a  three-form.  Any solution of the torsion lifting
equation
\[
\mathscr R_{g,J}K=F_X
\]
necessarily satisfies
$K\in\mathcal Z^1_{G_1}(g,J).$

Hence the space
\begin{equation}\label{eq:W-J-lift}
\mathcal W_{g,J}
:=\mathscr R_{g,J}\bigl(\mathcal Z^1_{G_1}(g,J)\bigr)
\subset\Omega^3(M)
\end{equation}
is finite dimensional, and an infinitesimal lift of the projected evolution  with varying $J$
exists for the chosen $X$

if and only if
\begin{equation}\label{eq:finite-dimensional-lift-condition}
F_X\in\mathcal W_{g,J}.
\end{equation}
In particular, if $\mathcal Z^1_{G_1}(g,J)=0$, then no nontrivial $J$-motion
can absorb the torsion defect: one must have $F_X=0$ and $K=0$.
\end{prop}

\begin{proof}
Since $\alpha_c\in\mathscr V_J\subset\Gamma(\mathcal E_J)$, its induced
torsion variation $\tau_c$ is a three-form.  The remaining terms in $F_X$
are also three-forms.  Applying $q$ to the lifting equation and using
\eqref{eq:S-qR} therefore gives
$\mathscr S_{g,J}K=0$.  The finite-dimensionality follows from
Corollary~\ref{cor:finite-dimensional-J-G1-deformations}, and the remaining
claims are immediate.
\end{proof}
\begin{remark}
\label{rem:GRF-lift-analytic-problem}
 Theorem~\ref{thm:exact-GRF-lift-constrained} is an infinitesimal criterion,
not a general existence theorem for a time-dependent moving-$J$ lift.
  The two lifting
operators are both overdetermined elliptic.  Indeed the Killing operator
\[
\mathscr K_g:Y\longmapsto\mathcal L_Yg
\]
has injective symbol $Y\mapsto2\,\xi\odot Y^\flat$, while
Proposition~\ref{prop:J-G1-overdetermined-elliptic} gives the corresponding
statement for the skew-torsion $J$-equation. Thus, whenever either lifting equation has a distributional
solution with smooth right-hand side, elliptic regularity
makes the solution smooth. If $M$ is compact, the corresponding
homogeneous kernels are finite dimensional.
Exact solvability, however, requires the right-hand side of each lifting equation to lie in the range of the corresponding operator.
 
For the projected constrained flow the conclusion is stronger:
Proposition~\ref{prop:projected-J-lift-finite-dimensional} shows that the
allowable $J$-correction already lies in the finite-dimensional space
$\mathcal Z^1_{G_1}(g,J)$.  Hence the torsion defect must belong to the
finite-dimensional correction space $\mathcal W_{g,J}$. 
 Thus exact gauge equivalence requires additional compatibility
conditions that do not follow from the projected evolution
equation alone.
\end{remark}

The previous lifting criterion simplifies in the following two situations.
\begin{cor}
For the unprojected Bismut--Ricci evolution, the lifting equation of
Proposition \ref{prop:exact-moving-J-unprojected-lift} is automatically satisfied with
$X=\frac12\theta^\sharp$ and $K=0$ in each of the following cases.

\begin{enumerate}
\item \begin{sloppypar}
If $J$ is integrable and the metric is pluriclosed, the unprojected
Bismut--Ricci velocity is already tangent to the pluriclosed constraint and
\eqref{eq:defect-zero-fixed-J} is precisely the standard pluriclosed/GRF gauge
equivalence of \cite[Corollary~9.9 and Remark~9.11]
{GarciaFernandezStreets2020}.
\end{sloppypar}
\item If $(g,J)$ is strong ABHE, then $\alpha_B=0$.  Equation~\eqref{eq:deltaT-Lee}
and $dT=0$ give
$dd_g^*T=-\mathcal L_{\theta^\sharp}T$, so
$\mathfrak D_J=0$.  Hence the strong ABHE structure is stationary in the
Bismut--Ricci gauge and self-similar in generalized Ricci gauge.  If moreover
$\theta=0$, it is an actual fixed point of generalized Ricci flow.
\end{enumerate}
\end{cor}

This formulation should be compared with the complex case in
\cite[Proposition~9.8--Remark~9.11]{GarciaFernandezStreets2020}: there the
metric and torsion defects both vanish identically and the moving complex
structure satisfies
\[
\partial_tJ=\frac12\mathcal L_{\theta^\sharp}J.
\]
In the almost complex setting the new operator $\mathscr R_{g,J}$ measures
exactly how much extra motion of $J$ is required to keep the characteristic
torsion on the generalized Ricci trajectory.  This is closely analogous in
spirit to the strong Sasaki-with-torsion flow, where a torsion-preserving
evolution is proved to be gauge-equivalent to generalized Ricci flow
\cite{BrienzaFinoFowdarGrantcharov2026}.

 In real dimension six, Theorem~\ref{thr3.16-dimsix} shows that every non-integrable
almost-SKT structure is automatically almost-CYT. \sloppy{We can therefore apply
the available ACYT results under their respective additional hypotheses.} 
It has been proven in \cite[Theorem~6.5]{IvanovStanchev2023} that a compact ACYT six-manifold with closed characteristic
torsion is a generalized gradient Ricci soliton and that, for the canonical
potential, the field $V=\theta^\sharp-\nabla f$ is $\nabla^B$-parallel and
preserves the full $SU(3)$-structure.
  {It has been further proven that} 
on a compact ACYT six-manifold with co-closed Lee form the characteristic
curvature is an $SU(3)$-instanton if and only if the torsion is
$\nabla^B$-parallel; the same equivalence holds in the balanced ACYT case
without compactness \cite{IvanovUgarte2025}.  These results are consistent
with Corollaries~\ref{cor:balanced-ABHE-GRF-fixed} and
\ref{cor:coclosed-Lee-ABHE-fixed}, while providing substantially stronger
\(SU(3)\)-specific curvature information.

\begin{remark} As in \cite[Theorem~3.5]{KennonStreets}, the strong ABHE equations imply
the steady generalized Ricci soliton identities;
equations \eqref{eq:RicB-Lee}--\eqref{eq:Ricci-comparison} give the corresponding local formulas in our conventions. The genuinely
non-integrable information relevant here is carried by the parallel field
$V$ and by the obstruction~\eqref{eq:LVJ-obstruction} to preserving $J$ under the soliton gauge.
\end{remark}

\section{Further questions}

{ 
The previous results isolate several natural problems left open by the
finite-dimensional existence argument.

\begin{enumerate}
\item   \emph{Non-rigid elliptic dynamics.}
Proposition~~\ref{prop:OE-not-rigidity} shows that \emph{(OE)} does not imply pointwise
Nijenhuis rigidity. It is natural to investigate  almost-SKT
structures on compact manifolds  satisfying \emph{(OE)} for which
$\dim\mathscr V_J>1$, and to study long-time existence,
convergence after normalization, and stability of homothetic
solutions in this setting.
\item \emph{Non-stationary examples in higher dimension.}
\sloppy{Theorem~\ref{thr3.16-dimsix} and Corollary~\ref{cor3.17-dimsix} show that in real dimension six every
compact non-integrable almost-SKT structure gives a stationary
constrained Bismut--Ricci flow. On the other hand, Corollary~\ref{cor:dim8-OE-fails}
shows that \emph{(OE)} fails identically in real dimension eight.}
Thus the first dimension in which a genuinely non-integrable,
non-stationary example in the overdetermined-elliptic regime could
occur is real dimension ten. It is natural to ask for a compact
almost-SKT manifold of real dimension at least ten satisfying
\emph{(OE)} and
$P_J\bigl((\rho^B_{\omega_0})^{1,1}\bigr)\neq0.$
Such an example would give a genuinely non-stationary constrained
Bismut--Ricci flow.

\item \emph{From local to compact examples.}
The cohomogeneity-one family of
Proposition~\ref{prop:S5-cohom1-strong-ABHE} is strong ABHE with non-flat
characteristic connection, but only local.  It would be interesting to know
whether any solution of \eqref{eq:S5-strong-ODE} extends to a complete
almost Hermitian six-manifold, or closes up to a compact one.
\end{enumerate}

\smallskip

{\bf Acknowledgements}  The authors would like to thank Beatrice Brienza  and Jeff Streets for very  helpful comments and discussions.
The first author is partially supported by Project PRIN 2022 \lq \lq Real and Complex Manifolds: Geometry and Holomorphic Dynamics", by GNSAGA (INdAM), and by a grant from the Simons Foundation (\# 944448).

\end{document}